\documentclass[10pt,reqno]{amsart}
\usepackage[a4paper,textwidth=140mm,textheight=230mm,
 centering,headheight=12pt,headsep=8mm,footskip=9mm]{geometry}
\usepackage[T1]{fontenc}
\usepackage{lmodern,amsmath,amssymb,amsthm,mathtools,microtype}
\usepackage[dvipsnames]{xcolor}
\usepackage{tikz-cd,needspace}
\newcommand{\x}{\mathcal{X}}
\newcommand{\y}{\mathcal{Y}}
\newcommand{\m}{\mathcal{M}}
\newcommand{\f}{\mathcal{F}}

\definecolor{refblue}{RGB}{0,55,150}
\definecolor{citered}{RGB}{175,25,35}
\usepackage[colorlinks=true,linkcolor=refblue,citecolor=citered,
 urlcolor=refblue,pagebackref]{hyperref}
\renewcommand*{\backref}[1]{}
\renewcommand*{\backrefalt}[4]{%
 \ifcase #1\relax\or\space(Cited on p.~#2.)%
 \else\space(Cited on pp.~#2.)\fi}
\makeatletter
\renewcommand{\section}{\@startsection{section}{1}{\z@}%
 {2.5ex plus 1ex minus .2ex}{1.3ex plus .2ex}{\normalfont\bfseries}}
\renewcommand{\subsection}{\@startsection{subsection}{2}{\z@}%
 {2ex plus .8ex minus .2ex}{.8ex plus .2ex}{\normalfont\itshape}}
\makeatother
\numberwithin{equation}{section}
\newtheorem{theorem}{Theorem}[section]
\newtheorem{lemma}[theorem]{Lemma}
\newtheorem{proposition}[theorem]{Proposition}

\newtheorem{introtheorem}{Theorem}

\theoremstyle{definition}
\newtheorem{definition}[theorem]{Definition}
\newtheorem*{introdefinition}{Definition}
\newtheorem{notation}[theorem]{Notation}
\newtheorem{fact}[theorem]{Fact}
\newtheorem{remark}[theorem]{Remark}
\newtheorem{example}[theorem]{Example}

\newcommand{\rad}{\operatorname{rad}}
\newcommand{\soc}{\operatorname{soc}}
\newcommand{\topm}{\operatorname{top}}
\newcommand{\End}{\operatorname{End}}
\newcommand{\Hom}{\operatorname{Hom}}

\newcommand{\im}{\operatorname{im}}
\newcommand{\coker}{\operatorname{coker}}

\newcommand{\modu}{\operatorname{mod}}
\newcommand{\op}{\mathrm{op}}
\newcommand{\len}{\ell}
\newcommand{\CC}{\mathcal C}

\begin{document}

\title[Composition factor algebras and finite representation type]
{Composition factor algebras and finite representation type}
\hypersetup{pdftitle={Composition factor algebras and finite representation type},pdfauthor={}}

\author{Bo Le and Bin Zhu}

\begin{abstract}
An Artin algebra is called a composition factor algebra if its
indecomposable finitely generated modules are determined, up to isomorphism,
by their composition factors. We prove that every composition factor algebra
is representation-finite. This settles a conjecture of Auslander,
Reiten and Smal\o. We also give criteria for comma categories and
factorization categories to be of finite representation type.
\end{abstract}
\maketitle
\tableofcontents

\section{Introduction}
The composition factors of a finite-length module, counted with
multiplicities, are among its simplest invariants. In general, nonisomorphic indecomposable
modules may have the same composition factors. Let $\Lambda$ be an Artin algebra. For a finitely generated
$\Lambda$-module $M$ and a simple $\Lambda$-module $S$, we write
$[M:S]$ for the multiplicity of $S$ in a composition series of $M$. We call $\Lambda$ a \emph{composition factor algebra}
(see Definition~\ref{def:composition-factor-algebra}) if
$$
 [M:S]=[N:S]\text{ for every simple }S
 \quad\Longrightarrow\quad M\cong N
$$
for any two indecomposable finitely generated $\Lambda$-modules
$M,N$.

The works of Happel \cite{Happel} and Auslander and Reiten
\cite{AR} give conditions under which composition factors
determine indecomposable modules; see also the historical notes
in \cite[Chapter~IX, p.~333]{ARS}. Auslander, Reiten and Smal\o\ conjectured that every composition factor
algebra is representation-finite \cite[Conjecture~(4), p.~409]{ARS}.
Here representation-finite means that there are only finitely many
isomorphism classes of indecomposable finitely generated modules.
Blasco proved the conjecture for commutative, hereditary or
radical-square-zero Artin algebras \cite{Blasco}. Our main result
establishes it for arbitrary Artin algebras.

\begin{theorem}\label{thm:main}
Every composition factor algebra is representation-finite.
\end{theorem}

The proof of Theorem~\ref{thm:main} uses the following finiteness
criteria for comma categories (cf.\ \cite[II, \S6]{MacLane})
and factorization categories (cf.\ \cite[\S1]{MV}).
Here and throughout, $R$ is a commutative Artin ring.

\begin{introdefinition}[see Definition~\ref{def:map-category}]
Let $D$ be a division Artin $R$-algebra,
and let $\x,\y$ be Hom-finite Krull--Schmidt
$R$-linear categories. Fix additive $R$-linear functors
$G:\x\to\modu D$ and $F:\y\to\modu D$.
The \emph{comma category} $\m(G,F)=(G\downarrow F)$
is defined as follows.
\begin{itemize}
\item Objects are triples $(X,Y,u)$ with $X\in\x$,
$Y\in\y$ and a homomorphism of $D$-modules
$u:G(X)\to F(Y)$.
\item A morphism $(X,Y,u)\to(X',Y',u')$ is a pair $(f,g)$
with $f:X\to X'$ in $\x$ and $g:Y\to Y'$ in
$\y$ satisfying $F(g)\circ u=u'\circ G(f)$.
\end{itemize}
Identities and composition are defined componentwise.
\end{introdefinition}

The following theorem gives a finiteness criterion for comma categories.

\begin{introtheorem}[Theorem~\ref{thm:map-finite}]\label{thm:intro-comma}
In the setting above,   assume that $\x$ and $\y$
are representation-finite. If any two indecomposable objects
of $\m(G,F)$ with isomorphic components are isomorphic,
then $\m(G,F)$ is representation-finite.
\end{introtheorem}


Let $\Lambda$ be a basic composition factor algebra.
Choose a primitive idempotent $e$ and write $e'=1-e$,
$A=e\Lambda e$, $B=e'\Lambda e'$, $U=e\Lambda e'$ and
$V=e'\Lambda e$.
For $M\in\modu\Lambda$, we write $X=eM$ and $N=e'M$.
The $\Lambda$-action on $M$ induces homomorphisms of $A$-modules
$$
U\otimes_B N\longrightarrow X\longrightarrow\Hom_B(V,N),
\qquad N\in\modu B,\quad X\in\modu A.
$$
By Lemma~\ref{prop:structure}, $J(A)U=0=VJ(A)$.
Hence both $U\otimes_B N$ and $\Hom_B(V,N)$ are annihilated
by $J(A)$ and are naturally modules over the division ring
$D=A/J(A)$. The functors $U\otimes_B-$ and $\Hom_B(V,-)$, together with
the natural transformation induced by multiplication in $\Lambda$,
define a factorization category equivalent to $\modu\Lambda$
(see Proposition~\ref{prop:module-diagrams}).
We give the general definition below.

\begin{introdefinition}[see Definition~\ref{def:factorization}]
Let $A$ be a local Artin $R$-algebra, write $D=A/J(A)$, and
let $\CC$ be a Hom-finite Krull--Schmidt $R$-linear category.
Fix additive $R$-linear functors $G,F:\CC\to\modu D$ and a
natural transformation $\mu:G\to F$. We view $D$-modules as
$A$-modules through the quotient $A\to D$.
The \emph{factorization category} $\f_A(\mu;G,F)$
is defined as follows.
\begin{itemize}
\item Objects are tuples $(N,X,\beta,\alpha)$ with $N\in\CC$,
$X\in\modu A$ and homomorphisms of $A$-modules
$$
G(N)\xrightarrow{\beta}X\xrightarrow{\alpha}F(N),
\qquad \alpha\circ\beta=\mu_N.
$$
\item A morphism $(N,X,\beta,\alpha)\to(N',X',\beta',\alpha')$
is a pair $(f,g)$ with $f:N\to N'$ in $\CC$ and
$g:X\to X'$ in $\modu A$ satisfying
$g\circ\beta=\beta'\circ G(f)$ and
$F(f)\circ\alpha=\alpha'\circ g$.
\end{itemize}
Identities and composition are defined componentwise.
\end{introdefinition}

For factorization categories, the following analogue of
Theorem~\ref{thm:intro-comma} holds.

\begin{introtheorem}[Theorem~\ref{prop:factorization}]\label{thm:intro-factorization}
In the setting above, assume that $\CC$ and $\modu A$ are
representation-finite. If any two indecomposable objects of
$\f_A(\mu;G,F)$ with isomorphic $\CC$-components and
isomorphic $A$-components are isomorphic, then
$\f_A(\mu;G,F)$ is representation-finite.
\end{introtheorem}

We first pass to the reduced category of Definition~\ref{def:reduction},
whose objects are factorizations with zero composite.
By Lemma~\ref{lem:normal-form}, this construction induces a bijection
between the isomorphism classes of indecomposable objects. Two applications of
Theorem~\ref{thm:intro-comma}, using the socle and a semisimple
quotient of the middle term, then prove
Theorem~\ref{thm:intro-factorization}.

We apply these criteria to prove Theorem~\ref{thm:main} by
induction on the $R$-length of the algebra. After passage to a basic
algebra, Lemma~\ref{prop:structure} applied to the block decomposition
in \eqref{eq:localalgebra} gives
$$
UV=0,\qquad J(A)U=0=VJ(A),\qquad U\otimes_BV=0.
$$
These relations yield the factorization description above.
The local corner $A$ is a quotient of $\Lambda$, hence a composition
factor algebra. It is Nakayama by Proposition~\ref{prop:local-finite}. A suitable
proper quotient of $\Lambda$ gives representation-finiteness
of $B$ by induction.   Under the factorization description,
indecomposable objects with isomorphic components correspond
to indecomposable $\Lambda$-modules with the same composition
factors. Since $\Lambda$ is a composition factor algebra, the
corresponding modules, and hence the factorization objects, are isomorphic.
Theorem~\ref{thm:intro-factorization} therefore implies that
$\Lambda$ is representation-finite.

All the results in this paper were obtained in early October 2026.
During the final preparation of this manuscript, we became aware
of the work \cite{CCWX}, which has just appeared
on arXiv and proves the same conjecture.

The paper is organized as follows.
Section~\ref{sec:preliminaries} recalls the required facts about
Artin algebras and Krull--Schmidt categories.
Section~\ref{sec:subspace} proves weak positivity of the quadratic
form associated with a comma category and applies Simson's
criterion for finite prinjective type \cite[Theorem~1.6]{Simson}
to establish Theorem~\ref{thm:intro-comma}.
Section~\ref{sec:factorization} introduces factorization
categories, establishes their basic properties, and proves
Theorem~\ref{thm:intro-factorization} using
Theorem~\ref{thm:intro-comma}.
In Section~\ref{sec:structure}, we study composition factor algebras
and describe their module categories by factorizations. We then prove
Theorem~\ref{thm:main} and give examples.

\medskip
\phantomsection
\addcontentsline{toc}{subsection}{\protect\tocsubsection{}{}{Conventions}}
\noindent\textbf{Conventions.}
Throughout the paper, $R$ is a commutative Artin ring.
The conventions below apply unless stated otherwise.
Algebras are associative unital $R$-algebras that are finitely
generated as $R$-modules. Modules are finitely generated left modules.
We write $\modu\Lambda$ for the category of such $\Lambda$-modules
and $J(\Lambda)$ for the Jacobson radical of $\Lambda$.

All additive categories are essentially small and $R$-linear.
A category is \emph{Hom-finite} if its morphism spaces have finite
length over the base ring $R$. An additive category is \emph{Krull--Schmidt}
if every object is a finite direct sum of objects with local
endomorphism rings. Functors between $R$-linear categories are
assumed to be $R$-linear.

\section{Preliminaries}\label{sec:preliminaries}

\subsection{Artin algebras and modules}\label{sec:preli-artin}
We recall some basic notions and facts about Artin algebras
and their modules from \cite{ARS}.

For $M\in\modu \Lambda$, write $\rad M$ and $\soc M$ for its radical
and socle, respectively, and set $\topm M=M/\rad M$.
Thus $\rad M=J(\Lambda)M$ and $\topm M=M/J(\Lambda)M$. We denote composition length
over $\Lambda$ by $\len_\Lambda(M)$ and length over $R$ by $\len_R(M)$.

\begin{definition}\label{def:local-basic}
\begin{enumerate}
\item[\textup{(1)}] An algebra $\Lambda$ is \emph{local} if $\Lambda/J(\Lambda)$
is a division ring, and \emph{basic} if $\Lambda/J(\Lambda)$ is a product
of division rings.
\item[\textup{(2)}] A module is \emph{uniserial} if it has a unique
composition series.
\item[\textup{(3)}] An algebra is \emph{Nakayama} if its
indecomposable projective and indecomposable injective left
modules are uniserial.
\end{enumerate}
\end{definition}

\begin{fact}\label{fact:finite-length}
Let $M$ be a nonzero $\Lambda$-module.
\begin{enumerate}
\item[\textup{(1)}] The socle $\soc M$ is essential in $M$.   In particular,
$\soc M\cap N\ne0$ for every nonzero submodule $N$ of $M$.
\item[\textup{(2)}] If $\topm M$ is simple, then $\rad M$ is
the unique maximal submodule of $M$.
\end{enumerate}
Consequently, a module with simple top or simple socle is
indecomposable.
\end{fact}

Let $E_R$ be the direct sum of the injective envelopes of a
complete set of simple $R$-modules, and set
$$
\mathbb D=\Hom_R(-,E_R).
$$
Matlis duality $\mathbb D:\modu \Lambda\to\modu \Lambda^{\op}$ is an exact
contravariant duality. It preserves $R$-length and interchanges
projective and injective modules.
In particular, it preserves indecomposability and induces a
bijection between the isomorphism classes of simple modules
on the two sides. Consequently,
$\len_{\Lambda^{\op}}(\mathbb D(M))=\len_\Lambda(M)$
for every $M\in\modu\Lambda$.

Let $A$ be a local algebra and put $D=A/J(A)$. We regard
$D$-modules as $A$-modules through the canonical quotient
map. This identifies $\modu D$ with the full subcategory of
semisimple $A$-modules.
Since $D$ is a division ring, every $V\in\modu D$ is a finite
direct sum of copies of ${}_D D$. In particular,
$\len_A(V)=\dim_D V$.

\subsection{Direct sums and endomorphisms}\label{sec:preli-categories}
We collect some basic facts about direct sums and endomorphism rings
for later use.


\begin{definition}\label{def:representation-finite}
A Krull--Schmidt category is \emph{representation-finite}
if it has only finitely many isomorphism classes of
indecomposable objects.
\end{definition}

\begin{lemma}[{cf.\ \cite[Proposition~2.9 and the paragraph following Corollary~4.4]{Krause}}]\label{lem:radical-matrices}
Let $\CC$ be a Hom-finite Krull--Schmidt category. Let
$X_1,\ldots,X_r$ be pairwise nonisomorphic indecomposable
objects of $\CC$, and set
$$
X=\bigoplus_{i=1}^r X_i^{n_i},\qquad E_i=\End_\CC(X_i),
\qquad n_i\ge1.
$$
The block-diagonal projection induces an isomorphism
$$
\End_\CC(X)/J(\End_\CC(X))
\cong\prod_{i=1}^r M_{n_i}(E_i/J(E_i)).
$$
More precisely, an endomorphism $f=(f_{ij})$ belongs to
$J(\End_\CC(X))$ if and only if
$f_{ii}\in M_{n_i}(J(E_i))$ for every $i$.
\end{lemma}

\begin{lemma}\label{lem:irreducible-endomorphism}
	Let $\Bbbk$ be a field and $\CC$ a Hom-finite
	Krull--Schmidt $\Bbbk$-linear category. Let
	$F:\CC\to\modu\Bbbk$ be a faithful $\Bbbk$-linear functor.
	Suppose that $f\in\End_\CC(M)$ satisfies $\phi(f)=0$
	for some monic irreducible polynomial
	$\phi\in\Bbbk[T]$ of degree $h$.
	Then $M$ has at most $\dim_{\Bbbk}F(M)/h$
	pairwise nonisomorphic indecomposable direct summands.
\end{lemma}

\begin{proof}
	Write
	$$
	M=\bigoplus_{i=1}^{t}M_i^{n_i},
	$$
	where the $M_i$ are pairwise nonisomorphic
	indecomposable objects. Set $E=\End_\CC(M)$ and
	$V=F(M)$. The action $a\cdot v=F(a)(v)$ makes
	$V$ a left $E$-module, which is faithful since
	$F$ is faithful.
	
	By Lemma~\ref{lem:radical-matrices}, $E$ has exactly
	$t$ isomorphism classes of simple left modules,
	represented by $S_1,\ldots,S_t$.
	Since $V$ is a faithful module over the finite-dimensional
	$\Bbbk$-algebra $E$, each $S_i$ occurs as a composition
	factor of $V$. Indeed, choosing a $\Bbbk$-basis $v_1,\ldots,v_d$ of $V$
	gives an injective $E$-linear map
	$E\to V^{\oplus d}$, $a\mapsto(av_1,\ldots,av_d)$.
	
	Since $\phi(f)=0$ and $\phi$ is irreducible,
	the action of $f$ on $S_i$ has minimal polynomial
	$\phi$. Thus $\dim_{\Bbbk}S_i\ge h$, and hence
	$
	th\le\sum_{i=1}^{t}\dim_\Bbbk S_i
	\le\dim_\Bbbk V.
	$
	This proves the assertion.
\end{proof}

For a ring $E$, let $E^\times$ denote its group of units. We shall use the following elementary lemma.

\begin{lemma}\label{lem:transitive}
Let $E_1,E_2$ be local Artin rings and $W$ a nonzero
$E_1$--$E_2$-bimodule.
Suppose that $E_1^\times\times E_2^\times$ acts transitively
on $W\setminus\{0\}$ by $(u,v)\cdot w=uwv^{-1}$.
Then $J(E_1)W=0=WJ(E_2)$.
\end{lemma}
\begin{proof}
For convenience, denote $J(E_1)$ by $J$ and set $Z=JW$. We have
$$
uZv^{-1}=(uJ)(Wv^{-1})=JW=Z
$$
for all $u\in E_1^\times$ and $v\in E_2^\times$.
Suppose that $Z\ne0$ and choose $0\ne z\in Z\subseteq W$. For every
$0\ne w\in W$, transitivity gives units $u,v$ such that
$w=uzv^{-1}$.   Hence $w\in Z$,   so $JW=W$.
Since $J$ is nilpotent, $JW=W$ implies $W=0$, a contradiction. Thus $J(E_1)W=0$.
Similarly, we have $WJ(E_2)=0$.
\end{proof}

The following lemma is well known. We include a proof for completeness.

\begin{lemma}\label{lem:nilpotent-quotient}
Let $\mathcal A$ be a Hom-finite Krull--Schmidt category and
$\mathcal B$ an additive category. Let $F:\mathcal A\to\mathcal B$ be a full,
essentially surjective additive $R$-linear functor. Suppose
that, for every $X\in\mathcal A$, the ideal
$$
I_X=\ker\bigl(\End_{\mathcal A}(X)
\longrightarrow\End_{\mathcal B}(F(X))\bigr)
$$
is nilpotent. Then \(\mathcal B\) is also Hom-finite and Krull–Schmidt. The functor \(F\) reflects isomorphisms and induces a bijection on isomorphism classes preserving finite direct sums and indecomposability.
\end{lemma}

\begin{proof}
	By fullness and $R$-linearity,
	$\Hom_{\mathcal B}(F(X),F(Y))$ is a quotient of
	$\Hom_{\mathcal A}(X,Y)$ as an $R$-module for all
	$X,Y\in\mathcal A$.
	Essential surjectivity therefore implies that
	$\mathcal B$ is Hom-finite.
	
	For each indecomposable object $X\in\mathcal A$, we have
	$
	\End_{\mathcal B}(F(X))
	\cong \End_{\mathcal A}(X)/I_X.
	$
	Since $\End_{\mathcal A}(X)$ is local and $I_X$ is
	nilpotent, this quotient is a nonzero local ring.
	Applying $F$ to indecomposable decompositions in
	$\mathcal A$ and using essential surjectivity shows
	that $\mathcal B$ is Krull--Schmidt.
	
	We next show that $F$ reflects isomorphisms.
	Let $u:X\to Y$ be a morphism such that $F(u)$ is invertible.
	By fullness, choose $v:Y\to X$ with $F(v)=F(u)^{-1}$.
	Then
	$$
	v\circ u=1_X+\varepsilon_X,\qquad
	u\circ v=1_Y+\varepsilon_Y,
	$$
	where $\varepsilon_X\in I_X$ and $\varepsilon_Y\in I_Y$.
	Since $\varepsilon_X$ and $\varepsilon_Y$ are nilpotent,
	both $v\circ u$ and $u\circ v$ are invertible.
	Hence $u$ is invertible.
	Essential surjectivity gives surjectivity on isomorphism
	classes. By fullness, any isomorphism $F(X)\cong F(Y)$
	lifts to a morphism $X\to Y$, which is invertible by the
	preceding argument.
	Thus $F$ induces a bijection on isomorphism classes.
	Preservation of finite direct sums follows from additivity.
	
	The isomorphism $
	\End_{\mathcal B}(F(X))
	\cong \End_{\mathcal A}(X)/I_X
	$ shows that $F$
	preserves indecomposability.
	Since $F$ is additive and $F(X)\cong0$ implies $X\cong0$,
	it also reflects indecomposability.
\end{proof}

\section{A finiteness criterion for comma categories}\label{sec:subspace}
Throughout this section, $D$ denotes a division Artin $R$-algebra.

\begin{definition}[{cf.\ \cite[II, \S6]{MacLane}}]\label{def:map-category}
For Hom-finite Krull--Schmidt categories $\x$ and $\y$,
fix additive functors
$$
G:\x\longrightarrow\modu D,\qquad
F:\y\longrightarrow\modu D.
$$
The \emph{comma category} $\m(G,F)=(G\downarrow F)$ is defined as follows.
\begin{itemize}
\item Objects are triples $(X,Y,u)$ with $X\in\x$,
$Y\in\y$ and a homomorphism of $D$-modules
$u:G(X)\to F(Y)$.
\item A morphism $(X,Y,u)\to(X',Y',u')$ is a pair
$(f,g)$ with $f:X\to X'$ in $\x$ and
$g:Y\to Y'$ in $\y$ satisfying
$$
F(g)\circ u=u'\circ G(f).
$$
\end{itemize}
Identities and composition are defined componentwise.
\end{definition}

\begin{remark}\label{rem:map-category-properties}
A morphism $(f,g)$ in $\m(G,F)$ is an isomorphism
if and only if both $f$ and $g$ are isomorphisms.
The category $\m(G,F)$ is additive and idempotent complete by
\cite[Proposition~3.5]{First}, with direct sums formed
componentwise. Each morphism space is an $R$-submodule of
the direct sum of the two component morphism spaces, so
$\m(G,F)$ is Hom-finite.   It is therefore
Krull--Schmidt by \cite[Corollary~4.4]{Krause}.
If $X,Y$ are indecomposable and $u\ne0$, then $(X,Y,u)$
is indecomposable because a nontrivial decomposition would
place $X$ and $Y$ in different summands and force $u=0$.
\end{remark}

We now state the main theorem of this section.

\begin{theorem}\label{thm:map-finite}
Let $\m(G,F)$ be as in Definition~\ref{def:map-category},
and assume that $\x$ and $\y$ are
representation-finite. Suppose that any two indecomposable
objects $(X,Y,u)$ and $(X',Y',u')$ of $\m(G,F)$
are isomorphic whenever $X\cong X'$ and $Y\cong Y'$.
Then $\m(G,F)$ is representation-finite.
\end{theorem}

\subsection{Quadratic forms and weak positivity}
Consider an additive
functor $G:\mathcal A\to\modu D$, where
$\mathcal A$ is a Hom-finite Krull--Schmidt category.
We say that $G$ is \emph{faithful on objects} if
$G(X)=0$ implies $X\cong0$.
The kernel ideal of $G$ is given by
$$
(\ker G)(X,Y)=\{f\in\Hom_{\mathcal A}(X,Y):G(f)=0\}.
$$
Set $\overline{\mathcal A}=\mathcal A/\ker G$, and let
$\pi:\mathcal A\to\overline{\mathcal A}$ be the canonical
quotient functor. Then $G$ induces a faithful additive
functor $\overline G:\overline{\mathcal A}\to\modu D$ with
$G=\overline G\circ\pi$.

For a Hom-finite Krull--Schmidt category $\mathcal B$ and an
additive functor $F:\mathcal B\to\modu D$, the quotient
functor induces
$$
\Pi:\m(G,F)\longrightarrow\m(\overline G,F),
\qquad (X,Y,u)\longmapsto(\pi(X),Y,u),
$$
with $\Pi(f,g)=(\pi(f),g)$ on morphisms.

\begin{lemma}\label{lem:faithful-image}
Assume that $G$ is faithful on objects.
Then $\overline{\mathcal A}$ is Hom-finite
and Krull--Schmidt. Both $\pi$ and $\Pi$ induce bijections on
isomorphism classes preserving direct sums and
indecomposability. Similarly, the induced functor
$\m(F,G)\to\m(F,\overline G)$ induces a
bijection on isomorphism classes preserving direct sums
and indecomposability.
\end{lemma}

\begin{proof}
For $X\in\mathcal A$, set
$$
I_X=\ker\bigl(\End_{\mathcal A}(X)\longrightarrow\End_D(G(X))\bigr).
$$
We first show that $I_X\subseteq J(\End_{\mathcal A}(X))$.
If $X$ is indecomposable, $I_X$ is a proper ideal because
$G(X)\ne0$, and hence is contained in the radical of the
local ring $\End_{\mathcal A}(X)$. In general, write $X$ as
a direct sum of indecomposables. If $G(f)=0$, the preceding
observation on each indecomposable summand and
Lemma~\ref{lem:radical-matrices} give
$f\in J(\End_{\mathcal A}(X))$.
Thus $I_X$ is nilpotent. Since $\pi$ is full and surjective
on objects, Lemma~\ref{lem:nilpotent-quotient} proves the
assertions about $\overline{\mathcal A}$ and $\pi$.

Since the quotient category has the same objects as
$\mathcal A$ and $\overline G\circ\pi=G$, the functor $\Pi$
is surjective on objects. To prove fullness, take a morphism
$(\overline f,g):\Pi(X,Y,u)\to \Pi(X',Y',u')$ and choose
$f:X\to X'$ with $\pi(f)=\overline f$. Since $(\overline f,g)$ is a morphism in $\m(\overline G,F)$, we have
$$
F(g)\circ u=u'\circ\overline G(\overline f)=u'\circ G(f).
$$
Thus $(f,g)$ is a morphism in $\m(G,F)$ and
$\Pi(f,g)=(\overline f,g)$, so $\Pi$ is full.
For $Z=(X,Y,u)$, the kernel of the induced map on
endomorphism rings is
$$
\{(f,0):f\in I_X\}.
$$
Since $I_X$ is nilpotent, this kernel is nilpotent.
By Lemma~\ref{lem:nilpotent-quotient}, $\Pi$ induces a
bijection on isomorphism classes preserving direct sums
and indecomposability.
\end{proof}

Let $\mathfrak m=\ker(R\to D)$ and $\Bbbk=R/\mathfrak m$.
Since $D$ is a division ring, $\mathfrak m$ is prime.
Since $R$ is a commutative Artinian ring, the prime ideal
$\mathfrak m$ is maximal by \cite[Proposition~8.1]{AM}.
Hence $\Bbbk=R/\mathfrak m$ is a field.
Since $\mathfrak mD=0$,   any finite set of $R$-module
generators of $D$ spans $D$ over $\Bbbk$. Thus $D$ is
finite-dimensional over $\Bbbk$.

Let $H:\mathcal E\to\modu D$ be a faithful additive
$R$-linear functor,   where $\mathcal E$ is Hom-finite.
For $r\in\mathfrak m$ and a morphism $f$ in $\mathcal E$,
we have $H(rf)=rH(f)=0$.
Faithfulness gives $rf=0$. Thus each morphism space of
$\mathcal E$ is annihilated by $\mathfrak m$ and, since it
has finite $R$-length, is finite-dimensional over $\Bbbk$.
Consequently, $\mathcal E$ is a Hom-finite $\Bbbk$-linear
category, and $H$ is a $\Bbbk$-linear functor.

\begin{definition}\label{def:map-form}
Let $\x,\y$ be nonzero representation-finite
Hom-finite Krull--Schmidt $\Bbbk$-linear categories, and let
$G:\x\to\modu D$ and $F:\y\to\modu D$
be faithful additive $\Bbbk$-linear functors.
\begin{enumerate}
\item[\textup{(1)}] For $X\in\x$ and $Y\in\y$, we define $q(X,Y)$ by
\begin{equation}\label{eq:map-form}
 q(X,Y)=\dim_{\Bbbk}\End_{\x}(X)
       +\dim_{\Bbbk}\End_{\y}(Y)
       -\dim_{\Bbbk}\Hom_D(G(X),F(Y)).
\end{equation}
\item[\textup{(2)}] The form $q$ is \emph{weakly positive} if
$q(X,Y)>0$ whenever $X$ and $Y$ are not both zero.
\end{enumerate}
Choose representatives $X_1,\ldots,X_r$ and $Y_1,\ldots,Y_s$
of the isomorphism classes of indecomposable objects in
$\x$ and $\y$, respectively. We write
$$
X=\bigoplus_{i=1}^{r}X_i^{n_i},\qquad
Y=\bigoplus_{j=1}^{s}Y_j^{m_j},
\qquad n_i,m_j\ge0.
$$
Then $q$ is an integral quadratic form in $(\mathbf n,\mathbf m)$.
\end{definition}

\begin{proposition}\label{prop:map-positivity}
Let $\x,\y,G,F,q$ be as in
Definition~\ref{def:map-form}.   Suppose that any two
indecomposable objects of $\m(G,F)$ with isomorphic
components are isomorphic. Then the following hold.
\begin{enumerate}
\item[\textup{(1)}] For any indecomposable objects
$X_1\in\x$ and $Y_1\in\y$, the rings
$\End_{\x}(X_1)$ and $\End_{\y}(Y_1)$
are division rings.
\item[\textup{(2)}] The quadratic form $q$ is weakly positive.
\end{enumerate}
\end{proposition}

\begin{proof}
\textup{(1)} Fix indecomposable objects
$X_1\in\x$ and $Y_1\in\y$, and set
$$
\begin{gathered}
E_X=\End_{\x}(X_1),\qquad
E_Y=\End_{\y}(Y_1),\\
U=G(X_1),\qquad V=F(Y_1),\qquad W=\Hom_D(U,V).
\end{gathered}
$$
Since $G$ and $F$ are faithful, $U,V$ are nonzero, so $W\ne0$.
For nonzero $u,v\in W$, the objects $(X_1,Y_1,u)$ and
$(X_1,Y_1,v)$ are indecomposable by
Remark~\ref{rem:map-category-properties}. They have the same
components, so the hypothesis gives an isomorphism
$(f,g):(X_1,Y_1,u)\to(X_1,Y_1,v)$, with
$f\in E_X^\times$ and $g\in E_Y^\times$. Then we have
$$
F(g)\circ u=v\circ G(f),\qquad
v=F(g)\circ u\circ G(f)^{-1}.
$$
Therefore $E_Y^\times\times E_X^\times$ acts transitively on
$W\setminus\{0\}$ by postcomposition and inverse precomposition.
Since $E_X,E_Y$ are finite-dimensional local $\Bbbk$-algebras,
Lemma~\ref{lem:transitive} gives
$$
J(E_Y)W=0=WJ(E_X).
$$

Let $f\in J(E_X)$. Then $w\circ G(f)=0$ for every $w\in W$.
If $G(f)\ne0$, choose $x\in U$ with $z=G(f)(x)\ne0$.
Extend $z$ to a $D$-basis $z,z_2,\ldots,z_t$ of $U$.
Choose $0\ne y\in V$ and define $w:U\to V$ $D$-linearly by
$w(z)=y$ and $w(z_i)=0$ for $i\ge2$.
Then $(w\circ G(f))(x)=y\ne0$, a contradiction.
Thus $G(f)=0$, and faithfulness of $G$ gives $f=0$.
Hence $J(E_X)=0$. Similarly, $J(E_Y)=0$.
Since $\End_{\x}(X_1)$ and $\End_{\y}(Y_1)$
are local rings with zero radical, they are division rings.

\textup{(2)} The inequality $q(X,Y)>0$ is immediate when
exactly one of $X$ and $Y$ is zero. We may therefore assume
that $X$ and $Y$ are nonzero and write
$$
X=\bigoplus_{i=1}^{r}X_i^{n_i},
\qquad
Y=\bigoplus_{j=1}^{s}Y_j^{m_j}.
$$
Set
$$
\begin{gathered}
	E_1=\End_{\y}(Y),\qquad E_2=\End_{\x}(X),
	\qquad W=\Hom_D(G(X),F(Y)),\\
	c=\dim_{\Bbbk}E_1+\dim_{\Bbbk}E_2,
	\qquad w=\dim_{\Bbbk}W.
\end{gathered}
$$
The space $W$ is an $E_1$--$E_2$-bimodule by postcomposition
and precomposition, and $q(X,Y)=c-w$. Suppose that $c\le w$.

Set $\m=\m(G,F)$ and
$$
\delta=\sum_{i=1}^{r}n_i+\sum_{j=1}^{s}m_j,
\qquad
N=\dim_{\Bbbk}G(X)+\dim_{\Bbbk}F(Y).
$$

For each positive integer $h$, the group
$\Gamma_h=\operatorname{GL}_h(E_1)\times
\operatorname{GL}_h(E_2)$ acts on the space $M_h(W)$
of $h\times h$ matrices with entries in $W$ by
$(g_1,g_2)\cdot u=g_1ug_2^{-1}$.
For $u\in M_h(W)$, write
$Z_u=(X^{\oplus h},Y^{\oplus h},u)\in\m$.
By the definition of morphisms in $\m$,
we have $Z_u\cong Z_v$ if and only if
$u$ and $v$ belong to the same $\Gamma_h$-orbit.

Let $\Omega_h$ be the union of those $\Gamma_h$-orbits
corresponding to objects $Z_u$ admitting an endomorphism
annihilated by some monic irreducible polynomial
$\phi\in\Bbbk[T]$ of degree $h$.
Let $\beta_h$ be the number of $\Gamma_h$-orbits
in $\Omega_h$, equivalently, the number of
isomorphism classes of objects $Z_u$ with
$u\in\Omega_h$.

For $u\in\Omega_h$, let $Z'=(X',Y',v)$ be an
indecomposable direct summand of $Z_u$.
Then
$$
X'\cong\bigoplus_{i=1}^{r}X_i^{n_i'},
\qquad
Y'\cong\bigoplus_{j=1}^{s}Y_j^{m_j'},
\qquad
0\le n_i'\le hn_i,\quad 0\le m_j'\le hm_j.
$$
By the hypothesis, each pair of multiplicity vectors
gives at most one isomorphism class of indecomposable
objects. Let $U_1,\ldots,U_t$ be pairwise nonisomorphic
representatives of the indecomposable direct summands
of $Z_u$, as $u$ ranges over $\Omega_h$. Then
$$
t\le
\prod_{i=1}^{r}(hn_i+1)\prod_{j=1}^{s}(hm_j+1)
\le(h\delta+1)^{r+s}.
$$

By Krull--Schmidt, for each $u\in\Omega_h$ we have
$$
Z_u\cong\bigoplus_{\ell=1}^{t}U_\ell^{a_\ell}
$$
with $0\le a_\ell\le h\delta$.

Consider the functor $H:\m\to\modu\Bbbk$
defined by $H(X',Y',v)=G(X')\oplus F(Y')$
and $H(f,g)=G(f)\oplus F(g)$.
This functor is faithful and $\Bbbk$-linear.
For $u\in\Omega_h$, we have $\dim_{\Bbbk}H(Z_u)=hN$.
By Lemma~\ref{lem:irreducible-endomorphism},
$Z_u$ has at most $N$ pairwise nonisomorphic
indecomposable direct summands. 
Hence at most $N$ of $a_1,\ldots,a_t$ are nonzero.
Thus the isomorphism class of $Z_u$ is determined by
at most $N$ pairs $(U_\ell,a_\ell)$, where
$1\le\ell\le t$ and $1\le a_\ell\le h\delta$.
Consequently,
\begin{equation}\label{eq:orbit-polynomial-bound}
	\beta_h\le(1+th\delta)^N
	\le(h\delta+1)^{N(r+s+1)}.
\end{equation}

\medskip\noindent\emph{Case (1). $\Bbbk$ is an infinite field.}
We now take $h=1$.
Since the zero endomorphism of every $Z_u$ is annihilated
by the monic irreducible polynomial $T$, we have
$\Omega_1=W$. Thus the action of
$E_1^\times\times E_2^\times$ on $W$ has finitely many orbits.
We now show that each orbit is contained in the zero set
of a nonzero polynomial on $W$.

Let $a=\dim_{\Bbbk}E_1$ and $b=\dim_{\Bbbk}E_2$.
Then $a+b=c$.
Choose $\Bbbk$-bases $e_1,\ldots,e_a$ of $E_1$,
$d_1,\ldots,d_b$ of $E_2$, and $f_1,\ldots,f_w$ of $W$.
Fix $u\in W$, and denote its orbit by $\mathcal O(u)$.
We have
$$
\mathcal O(u)=
\{F(g)\circ u\circ G(f):
g\in E_1^\times,\ f\in E_2^\times\}.
$$
For $g=\sum_{i=1}^{a}s_i e_i$ and
$f=\sum_{j=1}^{b}t_jd_j$, $\Bbbk$-linearity gives
$$
F(g)\circ u\circ G(f)
=\sum_{i=1}^{a}\sum_{j=1}^{b}
s_it_j\,F(e_i)\circ u\circ G(d_j).
$$
Here $u$ remains fixed, so the vectors
$F(e_i)\circ u\circ G(d_j)$ are fixed as $g,f$ vary.

Since $g$ is invertible, some coordinate $s_p$ is nonzero.
Replace $g,f$ by
$$
g'=s_p^{-1}g
=e_p+\sum_{\substack{i=1\\i\ne p}}^{a}
\frac{s_i}{s_p}e_i,
\qquad
f'=s_pf=\sum_{j=1}^{b}s_pt_jd_j.
$$
These are again units, and
$$
F(g')\circ u\circ G(f')
=s_p^{-1}s_p\,F(g)\circ u\circ G(f)
=F(g)\circ u\circ G(f).
$$
Thus the same orbit point has a representation in which
the coefficient of $e_p$ in $g'$ is $1$.
Once $p$ is chosen, the remaining coefficients
$x_i=s_i/s_p$ for $i\ne p$ and $y_j=s_pt_j$ still vary.
There are $(a-1)+b=c-1$ such variables.

For each $p\in\{1,\ldots,a\}$, define
$$
\begin{aligned}
	\Phi_{u,p}:\Bbbk^{a-1}\times\Bbbk^b
	&\longrightarrow W,\\
	(\mathbf x,\mathbf y)
	&\longmapsto
	F\left(e_p+\sum_{\substack{i=1\\i\ne p}}^{a}x_i e_i\right)
	\circ u\circ
	G\left(\sum_{j=1}^{b}y_jd_j\right),
\end{aligned}
$$
where $\mathbf x=(x_i)_{i\ne p}$ and
$\mathbf y=(y_1,\ldots,y_b)$.
The index $p$ specifies which coordinate is set to $1$.
Since every invertible $g$ has a nonzero coordinate,
the preceding normalization shows that
$$
\mathcal O(u)\subseteq
\bigcup_{p=1}^{a}\im\Phi_{u,p}.
$$

Expanding in the basis $f_1,\ldots,f_w$ of $W$ gives
$$
F(e_i)\circ u\circ G(d_j)
=\sum_{\ell=1}^{w}\gamma_{\ell ij}f_\ell,
\qquad \gamma_{\ell ij}\in\Bbbk.
$$
For a fixed $p$, the coordinates of $\Phi_{u,p}$ in the
chosen basis of $W$ are therefore given by
$$
\begin{aligned}
	\Phi_{u,p}(\mathbf x,\mathbf y)
	&=\sum_{\ell=1}^{w}r_\ell(\mathbf x,\mathbf y)f_\ell,\\
	r_\ell(\mathbf x,\mathbf y)
	&=\sum_{j=1}^{b}\gamma_{\ell pj}y_j
	+\sum_{\substack{i=1\\i\ne p}}^{a}\sum_{j=1}^{b}
	\gamma_{\ell ij}x_i y_j.
\end{aligned}
$$
Each $r_\ell$ is a polynomial of total degree at most $2$
in the $c-1$ remaining variables.
Since $c-1<w$, these $w$ polynomials are algebraically
dependent by \cite[VIII, Theorem~1.1]{Lang}.
Hence there is a nonzero polynomial
$P_{u,p}\in\Bbbk[T_1,\ldots,T_w]$ such that
$$
P_{u,p}\bigl(r_1(\mathbf x,\mathbf y),\ldots,
r_w(\mathbf x,\mathbf y)\bigr)=0.
$$
Using the chosen basis of $W$, we regard $P_{u,p}$ as a
polynomial function on $W$. Thus it vanishes on
$\im\Phi_{u,p}$.
Every point of $\mathcal O(u)$ lies in at least one of
these images, so $P_u=\prod_{p=1}^{a}P_{u,p}$ vanishes
on $\mathcal O(u)$.
Since the polynomial ring is a domain, $P_u\ne0$.

Choose representatives $u_1,\ldots,u_{\beta_1}$ of the
finitely many orbits. The product
$P_{u_1}\cdots P_{u_{\beta_1}}$ is nonzero and vanishes
at every point of $W\cong\Bbbk^w$.
This contradicts \cite[IV, Corollary~1.7]{Lang}.

\medskip\noindent\emph{Case (2). $\Bbbk$ is a finite field.}
Write $\Bbbk=\mathbb F_{p^n}$, where $p$ is prime
and $n\ge1$.

We now assume that $h$ is prime. Let $\mathcal I_h$
be the set of monic irreducible polynomials of degree
$h$ over $\Bbbk$. By \cite[Chapter~V, Exercise~22]{Lang},
we have $|\mathcal I_h|=(p^{nh}-p^n)/h$.

For $\phi\in\mathcal I_h$, let $M_\phi\in M_h(\Bbbk)$
be its companion matrix. Since $h>1$, we have
$\phi(0)\ne0$, so $M_\phi$ is invertible.
Via the central embeddings of $\Bbbk$ in $E_1,E_2$, set
$$
g_\phi=(M_\phi,M_\phi)\in\Gamma_h.
$$
The minimal polynomial of $M_\phi$ over $\Bbbk$ is $\phi$.

If $u\in M_h(W)$ is fixed by $g_\phi$, then
$M_\phi u=uM_\phi$, so $(M_\phi,M_\phi)$ is an
endomorphism of $Z_u$. Since $\phi(M_\phi)=0$,
this endomorphism is annihilated by $\phi$.
Thus, by the definition of $\Omega_h$, every fixed point
of $g_\phi$ on $M_h(W)$ belongs to $\Omega_h$.
For $g\in\Gamma_h$, write $\Omega_h^g$ for the fixed-point
set of $g$ in $\Omega_h$ and $C_{\Gamma_h}(g)$ for its
centralizer in $\Gamma_h$.

We now compute $|\Omega_h^{g_\phi}|$.
Fix $\phi\in\mathcal I_h$ and set
$$
C_\phi=\{A\in M_h(\Bbbk):M_\phi A=AM_\phi\},
\qquad d_\phi=\dim_{\Bbbk}C_\phi.
$$
Choose a $\Bbbk$-basis $f_1,\ldots,f_w$ of $W$.
Expanding each entry of $u=(u_{ij})\in M_h(W)$ in this basis
gives the unique expression
$$
u=\sum_{\ell=1}^{w}A_\ell f_\ell,\qquad
A_\ell=(a_{ij}^{(\ell)})\in M_h(\Bbbk),\qquad
u_{ij}=\sum_{\ell=1}^{w}a_{ij}^{(\ell)}f_\ell,
$$
where $A_\ell f_\ell$ denotes the matrix
$(a_{ij}^{(\ell)}f_\ell)_{ij}$.
Since the two scalar actions of $\Bbbk$ on $W$ agree,
$$
g_\phi\cdot u=u
\quad\Longleftrightarrow\quad M_\phi u=uM_\phi
\quad\Longleftrightarrow\quad
A_\ell\in C_\phi\quad(1\le\ell\le w).
$$
Since every fixed point of $g_\phi$ on $M_h(W)$
belongs to $\Omega_h$, the map
$$
C_\phi^{\oplus w}\longrightarrow\Omega_h^{g_\phi},
\qquad
(A_1,\ldots,A_w)\longmapsto
\sum_{\ell=1}^{w}A_\ell f_\ell
$$
is a $\Bbbk$-linear isomorphism. Hence
$\dim_{\Bbbk}\Omega_h^{g_\phi}=d_\phi w$.

For $i=1,2$, let
$$
C_{\phi,i}=\{B\in M_h(E_i):M_\phi B=BM_\phi\}.
$$
Similarly, $\dim_{\Bbbk}C_{\phi,i}
=d_\phi\dim_{\Bbbk}E_i$.
A pair $(g_1,g_2)\in\Gamma_h$ commutes with $g_\phi$
if and only if $g_iM_\phi=M_\phi g_i$ for $i=1,2$.
Hence
$$
C_{\Gamma_h}(g_\phi)
=\bigl(C_{\phi,1}\cap\operatorname{GL}_h(E_1)\bigr)
\times\bigl(C_{\phi,2}\cap\operatorname{GL}_h(E_2)\bigr).
$$
Since $|\Bbbk|=p^n$ and
$\dim_{\Bbbk}E_1+\dim_{\Bbbk}E_2=c$, we obtain
$$
|\Omega_h^{g_\phi}|=p^{nd_\phi w},
\qquad
|C_{\Gamma_h}(g_\phi)|
\le |C_{\phi,1}|\,|C_{\phi,2}|=p^{nd_\phi c}.
$$
Since $w\ge c$, we have
$$
\frac{|\Omega_h^{g_\phi}|}{|C_{\Gamma_h}(g_\phi)|}
\ge p^{nd_\phi(w-c)}\ge1.
$$

If $g_\phi$ and $g_\psi$ are conjugate in $\Gamma_h$,
then $M_\phi$ and $M_\psi$ are conjugate in
$\operatorname{GL}_h(E_1)$. Since conjugation preserves
their minimal polynomials over $\Bbbk$, we have
$\phi=\psi$.

Consider the action of $\Gamma_h$ on $\Omega_h$.
By the Cauchy--Frobenius formula
\cite[Theorem~3.22]{Rotman}, we have
$$
\beta_h
=\sum_{[g]}\frac{|\Omega_h^g|}{|C_{\Gamma_h}(g)|}
\ge\sum_{\phi\in\mathcal I_h}
\frac{|\Omega_h^{g_\phi}|}{|C_{\Gamma_h}(g_\phi)|}
\ge|\mathcal I_h|=\frac{p^{nh}-p^n}{h},
$$
where the first sum is over the conjugacy classes of
$\Gamma_h$.
Together with \eqref{eq:orbit-polynomial-bound}, this yields
$$
\frac{p^{nh}-p^n}{h}
\le(h\delta+1)^{N(r+s+1)},
$$
which is impossible as $h$ tends to infinity through
the primes.

This contradiction proves that $c>w$ in both cases.
\end{proof}

\subsection{From positivity to finite type}
We now use Simson's criterion for finite prinjective type.
We first recall it for triangular matrix algebras and then
realize $\m(G,F)$ as a prinjective category.

For division $\Bbbk$-algebras $D_1,D_2$ and a finite-dimensional
$D_1$--$D_2$-bimodule $V$, we write $\dim_{D_1}V$ for its
dimension as a left $D_1$-vector space and $\dim V_{D_2}$ for
its dimension as a right $D_2$-vector space.

Let $\Lambda_1,\Lambda_2$ be nonzero basic finite-dimensional
$\Bbbk$-algebras, and let $N$ be a finite-dimensional
$\Lambda_1$--$\Lambda_2$-bimodule.
Consider the triangular matrix algebra
$$
\Gamma=\begin{pmatrix}\Lambda_1&N\\0&\Lambda_2\end{pmatrix}.
$$
Choose complete sets of primitive orthogonal idempotents
$e_1,\ldots,e_r$ of $\Lambda_1$ and
$\eta_1,\ldots,\eta_s$ of $\Lambda_2$.
In this subsection, unless explicitly stated otherwise, modules
over $\Lambda_1$, $\Lambda_2$ and $\Gamma$ are right modules.

Following \cite[p.~96]{Simson}, we write $\operatorname{prin}(\Gamma)$
for Simson's category $\operatorname{prin}(\Gamma)^{\Lambda_1}_{\Lambda_2}$.
This is the full subcategory of finitely generated right
$\Gamma$-modules represented by triples
$(P,I,\varphi)$, where $P$ is a projective right $\Lambda_1$-module,
$I$ is an injective right $\Lambda_2$-module, and
$$
\varphi:P\otimes_{\Lambda_1}N\longrightarrow I
$$
is a homomorphism of right $\Lambda_2$-modules.

The corresponding right action on
$P\oplus I$ is
$$
(p,i)\begin{pmatrix}a&n\\0&b\end{pmatrix}
=\bigl(pa,\varphi(p\otimes n)+ib\bigr).
$$
A morphism $(P,I,\varphi)\to(P',I',\varphi')$ is a pair $(s,t)$
consisting of a homomorphism $s:P\to P'$ of right $\Lambda_1$-modules
and a homomorphism $t:I\to I'$ of right $\Lambda_2$-modules such that
$t\circ\varphi=\varphi'\circ(s\otimes1_N)$.
For $\mathbf n=(n_1,\ldots,n_r)\in\mathbb Z^r$ and
$\mathbf m=(m_1,\ldots,m_s)\in\mathbb Z^s$, define
\begin{equation}\label{eq:simson-form}
\begin{aligned}
 q_\Gamma(\mathbf n,\mathbf m)
 &=\sum_{i=1}^{r}\sum_{i'=1}^{r}\dim_{\Bbbk}(e_i\Lambda_1 e_{i'})\,n_i n_{i'}\\
 &\quad+\sum_{j=1}^{s}\sum_{j'=1}^{s}\dim_{\Bbbk}(\eta_j\Lambda_2\eta_{j'})\,m_j m_{j'}
\\
 &\quad-\sum_{i=1}^{r}\sum_{j=1}^{s}\dim_{\Bbbk}(e_iN\eta_j)\,n_i m_j.
\end{aligned}
\end{equation}
The form $q_\Gamma$ is \emph{weakly positive} if
$q_\Gamma(\mathbf n,\mathbf m)>0$ for every nonzero
$(\mathbf n,\mathbf m)\in\mathbb Z_{\ge0}^{r+s}$.

\begin{proposition}[{cf.\ \cite[Theorem~1.6]{Simson}}]\label{prop:simson-finite}
Assume that
\begin{enumerate}
\item[\textup{(i)}] $e_i\Lambda_1e_i$ and
$\eta_j\Lambda_2\eta_j$ are division rings for every $i,j$.
\item[\textup{(ii)}] The left $\Lambda_1$-module $N\eta_j$ and
the right $\Lambda_2$-module $e_iN$ are faithful for every $i,j$.
\end{enumerate}
If $q_\Gamma$ is weakly positive, then
$\operatorname{prin}(\Gamma)$ is representation-finite.
\end{proposition}

\begin{proof}
Conditions \textup{(i)} and \textup{(ii)} say precisely that
$\Gamma$ is piecewise peak in the sense of \cite[p.~97]{Simson}.
The algebra $\Gamma$ is Artinian because it is finite-dimensional
over $\Bbbk$. It also satisfies a polynomial identity: if
$d=\dim_{\Bbbk}\Gamma$, the nonzero noncommutative polynomial
$$
s_{d+1}(x_1,\ldots,x_{d+1})
=\sum_{\sigma\in\mathfrak S_{d+1}}\operatorname{sgn}(\sigma)
x_{\sigma(1)}\cdots x_{\sigma(d+1)}
$$
vanishes on $\Gamma$ because its evaluation is an alternating
$(d+1)$-linear map on a $d$-dimensional vector space.
Thus $\Gamma$ is a basic Artinian bipartite piecewise peak
PI-ring, as required in \cite[Theorem~1.6]{Simson}.

Write
$(\varepsilon_1,\ldots,\varepsilon_{r+s})
=(e_1,\ldots,e_r,\eta_1,\ldots,\eta_s)$ and set
$f_\ell=\dim_{\Bbbk}(\varepsilon_\ell\Gamma\varepsilon_\ell)$.
For $\ell\ne m$, let
$$
d_{\ell m}=\dim(\varepsilon_\ell\Gamma\varepsilon_m)
_{\varepsilon_m\Gamma\varepsilon_m},
\qquad
d'_{\ell m}=\dim_{\varepsilon_\ell\Gamma\varepsilon_\ell}
(\varepsilon_\ell\Gamma\varepsilon_m).
$$
Then
$$
\dim_{\Bbbk}(\varepsilon_\ell\Gamma\varepsilon_m)
=f_m d_{\ell m}=f_\ell d'_{\ell m}.
$$
In particular, $f_m d_{\ell m}=f_\ell d'_{\ell m}$ is the
balancing relation for the valuation used by Simson.
In his notation, $\ell\prec m$ means that $\ell\ne m$ and
$\varepsilon_\ell\Gamma\varepsilon_m\ne0$.
Consequently, \eqref{eq:simson-form} has diagonal weights
$f_\ell$ and off-diagonal weights $f_m d_{\ell m}$,
with positive signs for blocks of $\Lambda_1$ and $\Lambda_2$
and negative signs for blocks of $N$.
Thus \eqref{eq:simson-form} is precisely Simson's
weighted quadratic form (1.7), with the two positive
coefficients corrected as explained in
Remark~\ref{rem:simson-conventions}.
The implication $(c)\Rightarrow(a)$ in
\cite[Theorem~1.6]{Simson} proves the assertion.
\end{proof}
\begin{remark}\label{rem:simson-conventions}
Formula~(1.7) in \cite[p.~99]{Simson} contains a
typographical error in its two positive off-diagonal sums.
In Simson's notation \cite[(1.5), p.~98]{Simson},
$d_{ij}$ denotes the
\emph{right} dimension of $\varepsilon_i\Gamma\varepsilon_j$
over $\varepsilon_j\Gamma\varepsilon_j$, while $d'_{ij}$
is its left dimension over $\varepsilon_i\Gamma\varepsilon_i$.
Hence
$$
\dim_{\Bbbk}(\varepsilon_i\Gamma\varepsilon_j)
=f_jd_{ij}=f_id'_{ij}.
$$
The coefficients printed as $f_id_{ij}$ and $f_pd_{pq}$ in
the two positive sums of (1.7) must therefore be
$f_jd_{ij}$ and $f_qd_{pq}$, respectively.
These coefficients agree with the Euler-form comparison in
\cite[p.~129]{Simson}.
\end{remark}

\begin{lemma}\label{lem:simson-criterion}
Let $\x,\y,G,F,q$ be as in
Definition~\ref{def:map-form}, with indecomposable
representatives $X_1,\ldots,X_r$ and $Y_1,\ldots,Y_s$.
Assume that
\begin{enumerate}
\item[\textup{(1)}] $\End_{\x}(X_i)$ and
$\End_{\y}(Y_j)$ are division rings for every $i,j$.
\item[\textup{(2)}] The quadratic form $q$ is weakly positive.
\end{enumerate}
Then $\m(G,F)$ is representation-finite.
\end{lemma}

\begin{proof}
We write $D_{\Bbbk}=\Hom_{\Bbbk}(-,\Bbbk)$ and define
$$
\begin{gathered}
X_0=\bigoplus_{i=1}^{r}X_i,\qquad Y_0=\bigoplus_{j=1}^{s}Y_j,\qquad
\Lambda_1=\End_{\x}(X_0),\qquad
\Lambda_2=\End_{\y}(Y_0),\\
M=\Hom_D(G(X_0),F(Y_0)),\qquad N=D_{\Bbbk}(M),\qquad
\Gamma=\begin{pmatrix}\Lambda_1&N\\0&\Lambda_2\end{pmatrix}.
\end{gathered}
$$
The $\Lambda_2$--$\Lambda_1$-bimodule structure on $M$ is given by
$$
b\cdot u\cdot a=F(b)\circ u\circ G(a)
\qquad(a\in\Lambda_1,\ b\in\Lambda_2,\ u\in M).
$$
Thus $N$ is a $\Lambda_1$--$\Lambda_2$-bimodule.
Let $e_i\in\Lambda_1$ and $\eta_j\in\Lambda_2$
be the primitive idempotents corresponding to $X_i$ and $Y_j$.

We now construct an equivalence
\begin{equation}\label{eq:map-prinjective}
 \Phi:\m(G,F)\xrightarrow{\sim}\operatorname{prin}(\Gamma).
\end{equation}
For $X\in\x$ and $Y\in\y$, we define
$$
P_X=\Hom_{\x}(X_0,X),\qquad
I_Y=D_{\Bbbk}(\Hom_{\y}(Y,Y_0)).
$$
For morphisms $f:X\to X'$ and $g:Y\to Y'$, we define
$$
P(f)(p)=f\circ p,\qquad
I(g)(\xi)(h')=\xi(h'\circ g)
$$
for $p\in P_X$, $\xi\in I_Y$ and
$h'\in\Hom_{\y}(Y',Y_0)$.
These functors give equivalences from $\x$ to the finitely
generated projective right $\Lambda_1$-modules and from $\y$
to the finitely generated injective right $\Lambda_2$-modules.

For $u:G(X)\to F(Y)$, we define
$\varphi_u:P_X\otimes_{\Lambda_1}N\to I_Y$ by
\begin{equation}\label{eq:phi-u}
\varphi_u(p\otimes\lambda)(h)
=\lambda\!\bigl(F(h)\circ u\circ G(p)\bigr)
\end{equation}
for $p\in P_X$, $\lambda\in N$ and
$h\in\Hom_{\y}(Y,Y_0)$.
This formula is balanced over $\Lambda_1$ and defines a
homomorphism of right $\Lambda_2$-modules.
We define $\Phi$ on objects and morphisms by
$$
\Phi(X,Y,u)=(P_X,I_Y,\varphi_u),\qquad
\Phi(f,g)=(P(f),I(g)).
$$
By \eqref{eq:phi-u}, for $f:X'\to X$ and $g:Y\to Y'$, we have
$$
\varphi_{F(g)\circ u\circ G(f)}
=I(g)\circ\varphi_u\circ(P(f)\otimes1_N).
$$
The displayed identity gives naturality and shows that $\Phi$ is
well defined on morphisms. Functoriality follows from that of
$P$ and $I$.

For $X=X_0$ and $Y=Y_0$, we have
$P_{X_0}=\Lambda_1$ and $I_{Y_0}=D_{\Bbbk}(\Lambda_2)$.
Under $\Lambda_1\otimes_{\Lambda_1}N\cong N$, the standard
tensor--dual identifications are
$$
\Hom_{\Lambda_2}(N,D_{\Bbbk}(\Lambda_2))
\cong D_{\Bbbk}(N\otimes_{\Lambda_2}\Lambda_2)
\cong D_{\Bbbk}(N)\cong M.
$$
They send $\psi$ to the element of $M$ corresponding to the
functional $\lambda\mapsto\psi(\lambda)(1_{\Lambda_2})$.
Since $\varphi_u(\lambda)(1_{\Lambda_2})=\lambda(u)$,
they send $\varphi_u$ to $u$.
Every object of $\x$ is a direct summand of a finite direct
sum of copies of $X_0$, and similarly for $\y$ and $Y_0$.
Additivity and naturality therefore give a natural isomorphism
\begin{equation}\label{eq:prinjective-map}
 \Hom_{\Lambda_2}(P_X\otimes_{\Lambda_1}N,I_Y)
 \cong\Hom_D(G(X),F(Y)),\qquad \varphi_u\longmapsto u.
\end{equation}

For objects $(X,Y,u)$ and $(X',Y',u')$ and morphisms
$f:X\to X'$ and $g:Y\to Y'$, naturality and
\eqref{eq:prinjective-map} give
$$
F(g)\circ u=u'\circ G(f)
\quad\Longleftrightarrow\quad
I(g)\circ\varphi_u
=\varphi_{u'}\circ(P(f)\otimes1_N).
$$
Since $P$ and $I$ are fully faithful, $\Phi$ is fully faithful.
Every object of $\operatorname{prin}(\Gamma)$ is isomorphic
to one of the form $(P_X,I_Y,\psi)$, and
\eqref{eq:prinjective-map} gives a unique $u$ with
$\psi=\varphi_u$.
Thus $\Phi$ is essentially surjective, and hence an equivalence.

The algebras $\Lambda_1$ and $\Lambda_2$ are basic and
finite-dimensional over $\Bbbk$. Their primitive diagonal corner
rings are $\End_{\x}(X_i)$ and
$\End_{\y}(Y_j)$, which are division rings by hypothesis.

For each $j$, the right $\Lambda_1$-module
$$
\eta_jM=\Hom_D(G(X_0),F(Y_j))
$$
is faithful. In fact, if $a\in\Lambda_1$ annihilates all these maps,
then $u\circ G(a)=0$ for every $D$-linear map
$u:G(X_0)\to F(Y_j)$. If $G(a)(x)\ne0$, extend this vector
to a $D$-basis of $G(X_0)$ and choose $u$ sending it to a
nonzero vector of $F(Y_j)$. This contradicts $u\circ G(a)=0$.
Hence $G(a)=0$,   and faithfulness of $G$ gives $a=0$.
Similarly, each left $\Lambda_2$-module
$Me_i=\Hom_D(G(X_i),F(Y_0))$ is faithful. Duality therefore
makes the left $\Lambda_1$-module $N\eta_j$ and the right
$\Lambda_2$-module $e_iN$ faithful.
Thus $\Gamma$ satisfies conditions \textup{(i)} and \textup{(ii)} of
Proposition~\ref{prop:simson-finite}.

By the choice of $e_i$ and $\eta_j$ as the projections onto
$X_i$ and $Y_j$, respectively, we have
$$
e_i\Lambda_1e_{i'}\cong\Hom_{\x}(X_{i'},X_i),\qquad
\eta_j\Lambda_2\eta_{j'}\cong\Hom_{\y}(Y_{j'},Y_j).
$$
For $X=\bigoplus_{i=1}^{r}X_i^{n_i}$ and
$Y=\bigoplus_{j=1}^{s}Y_j^{m_j}$, we have
$$
\begin{aligned}
\dim_{\Bbbk}\End_{\x}(X)
&=\sum_{i=1}^{r}\sum_{i'=1}^{r}
\dim_{\Bbbk}(e_i\Lambda_1e_{i'})\,n_i n_{i'},\\
\dim_{\Bbbk}\End_{\y}(Y)
&=\sum_{j=1}^{s}\sum_{j'=1}^{s}
\dim_{\Bbbk}(\eta_j\Lambda_2\eta_{j'})\,m_j m_{j'}.
\end{aligned}
$$
Moreover, $\eta_jMe_i\cong\Hom_D(G(X_i),F(Y_j))$ and
$e_iN\eta_j\cong D_{\Bbbk}(\eta_jMe_i)$, so we have
$$
\begin{aligned}
\dim_{\Bbbk}\Hom_D(G(X),F(Y))
&=\sum_{i=1}^{r}\sum_{j=1}^{s}
\dim_{\Bbbk}\Hom_D(G(X_i),F(Y_j))\,n_i m_j\\
&=\sum_{i=1}^{r}\sum_{j=1}^{s}
\dim_{\Bbbk}(e_iN\eta_j)\,n_i m_j.
\end{aligned}
$$
By \eqref{eq:map-form} and \eqref{eq:simson-form}, we have
$q(X,Y)=q_\Gamma(\mathbf n,\mathbf m)$.
Hence $q_\Gamma$ is weakly positive. By
Proposition~\ref{prop:simson-finite}, $\operatorname{prin}(\Gamma)$
is representation-finite, and the equivalence
\eqref{eq:map-prinjective} proves the assertion.
\end{proof}

We can now prove Theorem~\ref{thm:map-finite}.

\begin{proof}[Proof of Theorem~\ref{thm:map-finite}]
Let $\x'$ and $\y'$ be the full additive subcategories generated
by indecomposable $X$ with $G(X)\ne0$ and $Y$ with $F(Y)\ne0$,
respectively.
For $Z=(X,Y,u)\in\m(G,F)$, write
$X\cong X_0\oplus X_1$ and $Y\cong Y_0\oplus Y_1$,
where $G(X_0)=F(Y_0)=0$, $X_1\in\x'$ and $Y_1\in\y'$.
So we have
$$
Z\cong(X_0,0,0)\oplus(0,Y_0,0)\oplus(X_1,Y_1,u_1),
$$
where $u_1:G(X_1)\to F(Y_1)$ is induced by $u$.
Since $\x$ and $\y$ are representation-finite, there are only
finitely many indecomposables of the forms $(X,0,0)$ and
$(0,Y,0)$ up to isomorphism.
So it suffices to prove that the full subcategory
$\m(G|_{\x'},F|_{\y'})$ is representation-finite.
Since this subcategory satisfies the same hypothesis, we may
replace $\x,\y$ by $\x',\y'$ and assume that $G,F$ are faithful
on objects.
The conclusion is immediate if $\x=0$ or $\y=0$, so we assume
that both are nonzero.

Apply Lemma~\ref{lem:faithful-image} to both functors and set
$$
\overline{\x}=\x/\ker G,\qquad
\overline{\y}=\y/\ker F.
$$
The induced functors $\overline G,\overline F$ are faithful,
and both quotient categories remain nonzero, Hom-finite,
Krull--Schmidt and representation-finite.

We now show that any two indecomposable objects of
$\m(\overline G,\overline F)$ with isomorphic
components are isomorphic. Choose objects in $\m(G,F)$ lifting the two given
indecomposable objects.
By Lemma~\ref{lem:faithful-image}, these lifts are
indecomposable, and their components are isomorphic because
the component quotient functors induce bijections on
isomorphism classes. The hypothesis makes the lifts
isomorphic, so their images are isomorphic as well.

Set $\Bbbk=R/\ker(R\to D)$.
By the observation preceding Definition~\ref{def:map-form},
$D$ is finite-dimensional over $\Bbbk$, and the categories
$\overline{\x},\overline{\y}$ and functors
$\overline G,\overline F$ satisfy the assumptions of that definition.
 By Proposition~\ref{prop:map-positivity} and
Lemma~\ref{lem:simson-criterion},
$\m(\overline G,\overline F)$ is representation-finite.
By Lemma~\ref{lem:faithful-image}, $\m(G,F)$ is also
representation-finite.
\end{proof}

\section{Factorization categories}\label{sec:factorization}
\subsection{Factorizations of a natural transformation}
We use the factorization construction of MacPherson and Vilonen
\cite[\S1]{MV} in the following additive setting.

\begin{definition}\label{def:factorization}
Let $A$ be a local Artin $R$-algebra and put
$D=A/J(A)$. For a Hom-finite Krull--Schmidt
category $\CC$,   fix additive functors
$$
G,F:\CC\longrightarrow\modu D
$$
and a natural transformation $\mu:G\to F$.
We regard $D$-modules as semisimple $A$-modules through
$A\to D$. The \emph{factorization category}
$\f_A(\mu;G,F)$ is defined as follows.
\begin{itemize}
\item Objects are tuples $(N,X,\beta,\alpha)$ with $N\in\CC$,
$X\in\modu A$ and homomorphisms of $A$-modules
$$
G(N)\xrightarrow{\beta}X\xrightarrow{\alpha}F(N)
\qquad\text{such that}\qquad \alpha\circ\beta=\mu_N.
$$
\item A morphism
$(N,X,\beta,\alpha)\to(N',X',\beta',\alpha')$ is a pair
$(f,g)$ with $f:N\to N'$ in $\CC$ and
$g:X\to X'$ in $\modu A$ for which the following
diagram commutes:
$$
\begin{tikzcd}[column sep=large]
G(N)\arrow[r,"\beta"]\arrow[d,"G(f)"']
&X\arrow[r,"\alpha"]\arrow[d,"g"']&F(N)\arrow[d,"F(f)"]\\
G(N')\arrow[r,"\beta'"']&X'\arrow[r,"\alpha'"']&F(N')
\end{tikzcd}
$$
\end{itemize}
Identities and composition are defined componentwise.
\end{definition}

\begin{remark}\label{rem:factorization-basic}
A morphism $(f,g)$ in $\f_A(\mu;G,F)$ is an isomorphism if and
only if both $f$ and $g$ are isomorphisms. An object
$(N,X,\beta,\alpha)$ is zero if and only if $N=0$ and $X=0$.
For every $Y\in\modu A$, we have
$\End_{\f_A(\mu;G,F)}(0,Y,0,0)\cong\End_A(Y)$.
Thus $(0,Y,0,0)$ is indecomposable if and only if $Y$ is.
The tuple $(N,0,0,0)$ is an object of $\f_A(\mu;G,F)$ precisely
when $\mu_N=0$. In that case its endomorphism ring is
$\End_\CC(N)$, so it is indecomposable if and only if $N$ is.

If $N$ and $X$ are indecomposable and at least one of
$\beta,\alpha$ is nonzero, then $(N,X,\beta,\alpha)$ is
indecomposable. Indeed, a nontrivial decomposition would
place $N$ and $X$ in different summands and force both maps
to vanish.
\end{remark}

\begin{lemma}\label{lem:factorization-ks}
The category $\f_A(\mu;G,F)$ is a Hom-finite,
idempotent-complete $R$-linear additive category, with finite
direct sums formed componentwise. In particular, it is
Krull--Schmidt.
\end{lemma}

\begin{proof}
Under the canonical identifications supplied by additivity
of $G,F$, naturality gives
$\mu_{N_1\oplus N_2}=\mu_{N_1}\oplus\mu_{N_2}$.
Hence componentwise direct sums, injections and projections
make $\f_A(\mu;G,F)$ an additive category.

Let $(f,g)$ be an idempotent endomorphism of
$Z=(N,X,\beta,\alpha)$. Split $f$ and $g$ in $\CC$ and
$\modu A$, respectively, and write
$$
N\cong N_1\oplus N_2,\qquad X\cong X_1\oplus X_2,
\qquad f=1_{N_1}\oplus0,\quad g=1_{X_1}\oplus0.
$$
Since $(f,g)$ is an endomorphism of $Z$, we have
$g\circ\beta=\beta\circ G(f)$ and
$F(f)\circ\alpha=\alpha\circ g$.
These relations make $\beta$ and $\alpha$ block diagonal
in the chosen decompositions. Their diagonal
blocks satisfy
$$
\alpha_i\circ\beta_i=\mu_{N_i}
\qquad(i=1,2)
$$
by naturality of $\mu$. They therefore define objects
$Z_i=(N_i,X_i,\beta_i,\alpha_i)$ with $Z\cong Z_1\oplus Z_2$,
and $(f,g)$ is the projection onto $Z_1$.
Thus $\f_A(\mu;G,F)$ is idempotent complete.

For $Z=(N,X,\beta,\alpha)$ and $Z'=(N',X',\beta',\alpha')$,
the space $\Hom_{\f_A(\mu;G,F)}(Z,Z')$ is an $R$-submodule of
$\Hom_\CC(N,N')\oplus\Hom_A(X,X')$,   so $\f_A(\mu;G,F)$ is
Hom-finite. Since it is also idempotent complete, it is
Krull--Schmidt by \cite[Corollary~4.4]{Krause}.
\end{proof}

\subsection{Reduction and finiteness}
Throughout this subsection, let $A$ be a local Artin $R$-algebra
and write $D=A/J(A)$. Let $\mu:G\to F$ be a natural transformation
between additive functors $G,F:\CC\to\modu D$, where $\CC$ is
Hom-finite and Krull--Schmidt. We work with the category $\f_A(\mu;G,F)$.

We consider the kernel, image and cokernel functors
$$
K,I,C:\CC\longrightarrow\modu D,
$$
given by
$$
K(N)=\ker\mu_N,\qquad I(N)=\im\mu_N,\qquad
C(N)=\coker\mu_N.
$$
Naturality of $\mu$ induces the maps $K(f)$, $I(f)$ and
$C(f)$ for every $f:N\to N'$. Since every $D$-module is semisimple, we choose decompositions
of $G(N)$ and $F(N)$ for each $N$ in which $\mu_N$ has the form
$$
G(N)=K(N)\oplus I(N)
\xrightarrow{\;\mu_N=\left(\begin{smallmatrix}0&1\\0&0\end{smallmatrix}\right)\;}
I(N)\oplus C(N)=F(N).
$$

For an object $Z=(N,X,\beta,\alpha)$ of
$\f_A(\mu;G,F)$, set
$$
\iota=\beta|_{I(N)},\qquad
p=\operatorname{pr}_{I(N)}\circ\alpha:X\longrightarrow I(N),
$$
where $\operatorname{pr}_{I(N)}:I(N)\oplus C(N)\to I(N)$
is the projection onto the first summand. The relation
$\alpha\circ\beta=\mu_N$ gives $p\circ\iota=1_{I(N)}$.
Taking $Y=\ker p$, we obtain a decomposition
$$
X\cong I(N)\oplus Y.
$$
In these coordinates, the maps take the form
\begin{equation}\label{eq:factorization-normal-form}
\begin{aligned}
 \beta&=\begin{pmatrix}0&1\\b&0\end{pmatrix}:
 K(N)\oplus I(N)\longrightarrow I(N)\oplus Y,\\
 \alpha&=\begin{pmatrix}1&0\\0&a\end{pmatrix}:
 I(N)\oplus Y\longrightarrow I(N)\oplus C(N),
\end{aligned}
\end{equation}
with $a\circ b=0$.

\begin{definition}\label{def:reduction}
The category $\f_A(0;K,C)$ is called the
\emph{reduced category} of $\f_A(\mu;G,F)$.
The object
$$
(N,Y,b,a)\in\f_A(0;K,C)
$$
constructed above is called a \emph{reduced object} associated
with $Z=(N,X,\beta,\alpha)$.
\end{definition}

Conversely, for
$(N,Y,b,a)\in\f_A(0;K,C)$, set
$X=I(N)\oplus Y$ and define $\beta,\alpha$ by
\eqref{eq:factorization-normal-form}. Their composite
is $\mu_N$ because $a\circ b=0$.

\begin{lemma}\label{lem:normal-form}
The construction above induces a bijection between the
isomorphism classes of objects of $\f_A(\mu;G,F)$
and those of $\f_A(0;K,C)$, given by
$$
[(N,X,\beta,\alpha)]\longmapsto[(N,Y,b,a)].
$$
This bijection preserves direct sums and reflects zero objects.
In particular, it restricts to a bijection between the
isomorphism classes of indecomposable objects.
Corresponding objects have the same $\CC$-component and satisfy
$$
X\cong I(N)\oplus Y.
$$
\end{lemma}

\begin{proof}
Fix any choices of splittings for two objects and write them
in the coordinates of \eqref{eq:factorization-normal-form}.
Let $(f,g)$ be a morphism between them. In the chosen decompositions the naturality equation
$F(f)\circ\mu_N=\mu_{N'}\circ G(f)$ forces the lower-left
blocks of $G(f)$ and $F(f)$ to vanish and their blocks on
$I(N)$ to agree. The conditions
$g\circ\beta=\beta'\circ G(f)$ and
$F(f)\circ\alpha=\alpha'\circ g$ are expressed by the
following commutative diagram:
$$
\begin{tikzcd}[ampersand replacement=\&,column sep=huge,row sep=huge]
K(N)\oplus I(N)
\arrow[r,"{\left(\begin{smallmatrix}0&1\\b&0\end{smallmatrix}\right)}"]
\arrow[d,"{G(f)=\left(\begin{smallmatrix}k&u\\0&i\end{smallmatrix}\right)}"']
\&I(N)\oplus Y
\arrow[r,"{\left(\begin{smallmatrix}1&0\\0&a\end{smallmatrix}\right)}"]
\arrow[d,"g"]
\&I(N)\oplus C(N)
\arrow[d,"{F(f)=\left(\begin{smallmatrix}i&v\\0&c\end{smallmatrix}\right)}"]\\
K(N')\oplus I(N')
\arrow[r,"{\left(\begin{smallmatrix}0&1\\b'&0\end{smallmatrix}\right)}"']
\&I(N')\oplus Y'
\arrow[r,"{\left(\begin{smallmatrix}1&0\\0&a'\end{smallmatrix}\right)}"']
\&I(N')\oplus C(N')
\end{tikzcd}.
$$
Here $k=K(f)$, $i=I(f)$ and $c=C(f)$.
The diagram commutes if and only if
$$
g=\begin{pmatrix}i&v\circ a\\b'\circ u&h\end{pmatrix},
\qquad h\circ b=b'\circ k,\qquad a'\circ h=c\circ a.
$$
The last two equations say exactly that $(f,h)$ is a
morphism between the corresponding objects of
$\f_A(0;K,C)$.
If $f$ is an isomorphism, then $k,i,c$ are invertible.
Using $a\circ b=0$ and $h\circ b=b'\circ k$,   we obtain
$$
g\begin{pmatrix}1&0\\-b\circ k^{-1}\circ u&1\end{pmatrix}
=\begin{pmatrix}i&v\circ a\\0&h\end{pmatrix}.
$$
Thus $(f,g)$ is an isomorphism if and only if $(f,h)$ is.

We next check that the construction is well defined on
isomorphism classes. Let two choices of splittings for the same object
$Z=(N,X,\beta,\alpha)$ give reduced objects
$(N,Y_1,b_1,a_1)$ and $(N,Y_2,b_2,a_2)$.
Express $(1_N,1_X)$ in the corresponding normal-form
coordinates. Since its middle component is invertible, the
calculation above gives an isomorphism $h:Y_1\to Y_2$ with
$h\circ b_1=b_2$ and $a_2\circ h=a_1$. Thus $(1_N,h)$
is an isomorphism between the reduced objects. If two original
objects are isomorphic,   the same calculation applied to an
isomorphism between them shows that their reduced objects
are isomorphic.

The isomorphism criterion above also gives injectivity on
isomorphism classes. Every reduced object is obtained by
reducing its reconstruction, which gives surjectivity.

Since $K,I,C$ are additive, taking direct sums of the chosen
splittings shows that the bijection preserves direct sums.
A reduced object $(N,Y,b,a)$ is zero precisely when
$N\cong0$ and $Y=0$. By $X\cong I(N)\oplus Y$, this is
equivalent to $Z=(N,X,\beta,\alpha)$ being zero. If a reduced
object decomposes into two nonzero summands, their inverse
images give a decomposition of the original object, by
additivity and injectivity of the bijection. The reverse
implication follows from preservation of direct sums.
Thus indecomposable objects correspond.

\end{proof}

\begin{definition}\label{def:factorization-uniqueness}
We say that $\f_A(\mu;G,F)$ satisfies
\emph{property~$(P)$} if any two indecomposable objects
$Z=(N,X,\beta,\alpha)$ and $Z'=(N',X',\beta',\alpha')$
in $\f_A(\mu;G,F)$ are isomorphic whenever
$N\cong N'$ in $\CC$ and $X\cong X'$ in $\modu A$.
\end{definition}

We first relate the reduced category to two comma categories.
\begin{lemma}\label{lem:factorization-comma}
	Suppose that $K(T)=0$ or $C(T)=0$ for every indecomposable
	$T\in\CC$. Let $\CC_1$ and $\CC_2$ be the full additive
	subcategories generated by the indecomposable objects $T$
	with $K(T)\ne0$ and $K(T)=0$, respectively. Set
	$$
	\mathcal A=\m(K|_{\CC_1},\soc),
	\qquad \soc:\modu A\longrightarrow\modu D.
	$$
	There are additive functors
	$$
	\begin{aligned}
		Q&:\mathcal A\longrightarrow\modu D,\\
		\Theta&:\m(Q,C|_{\CC_2})\longrightarrow\f_A(0;K,C)
	\end{aligned}
	$$
	such that $\Theta$ induces a bijection on isomorphism
	classes of indecomposable objects.
\end{lemma}

\begin{proof}
	By assumption, $C|_{\CC_1}=0=K|_{\CC_2}$.
	For $T_i\in\CC_i$, we use the canonical identifications
	$$
	K(T_1\oplus T_2)\cong K(T_1),\qquad
	C(T_1\oplus T_2)\cong C(T_2).
	$$
	
	For $Z=(T_1,Y,b)\in\mathcal A$, regard
	$b:K(T_1)\to\soc Y$ as a map into $Y$ and set
	$$
	Q(Z)=Y/(J(A)Y+\im b).
	$$
	Let $\rho_Z:Y\to Q(Z)$ be the quotient map.
	For a morphism $(f_1,g):Z\to Z'$ in $\mathcal A$,
	where $Z'=(T_1',Y',b')$, we have
	$
	g\circ b=b'\circ K(f_1),
	$
	and so $g(\im b)\subseteq\im b'$.
	Since $g$ is $A$-linear, we have
	$g(J(A)Y)\subseteq J(A)Y'$.
	Hence
	$$
	g(J(A)Y+\im b)\subseteq J(A)Y'+\im b',
	$$
	so $(f_1,g)$ induces a $D$-linear map
	$Q(f_1,g):Q(Z)\to Q(Z')$ satisfying
	$$
	Q(f_1,g)\circ\rho_Z=\rho_{Z'}\circ g.
	$$
	These assignments define an additive functor
	$Q:\mathcal A\to\modu D$.

	For an object $(Z,T_2,\overline a)$ of
	$\m(Q,C|_{\CC_2})$, with $Z=(T_1,Y,b)$, set
	$$
	\Theta(Z,T_2,\overline a)
	=(T_1\oplus T_2,Y,b,\overline a\circ\rho_Z).
	$$
	On morphisms, set
	$$
	\Theta((f_1,g),f_2)=(f_1\oplus f_2,g).
	$$
	The relation $\rho_Z\circ b=0$ and the naturality identity
	for $\rho_Z$ show that these formulas define an additive
	functor
	$$
	\Theta:\m(Q,C|_{\CC_2})\longrightarrow\f_A(0;K,C).
	$$
	
	We first show that $\Theta$ is essentially surjective.
	Let $W=(N,Y,b,a)\in\f_A(0;K,C)$.
	By Krull--Schmidt, write $N\cong T_1\oplus T_2$ with
	$T_i\in\CC_i$. Since $J(A)K(T_1)=0$, the $A$-linear map
	$b:K(T_1)\to Y$ has image in $\soc Y$.
	Thus $Z=(T_1,Y,b)$ is an object of $\mathcal A$.
	The map $a:Y\to C(T_2)$ vanishes on $J(A)Y+\im b$
	because $J(A)C(T_2)=0$ and $a\circ b=0$.
	Hence there is a unique $D$-linear map
	$\overline a:Q(Z)\to C(T_2)$ such that
	$$
	a=\overline a\circ\rho_Z.
	$$
	Thus
	$
	\Theta(Z,T_2,\overline a)
	=(T_1\oplus T_2,Y,b,a)\cong W.
	$
	
	Now let $U=(Z,T_2,\overline a)$ and
	$U'=(Z',T_2',\overline a')$, where
	$Z=(T_1,Y,b)$ and $Z'=(T_1',Y',b')$.
	Suppose that $(f,g):\Theta(U)\to\Theta(U')$ is an isomorphism,
	and write
	$$
	f=\begin{pmatrix}f_1&u\\v&f_2\end{pmatrix}:
	T_1\oplus T_2\longrightarrow T_1'\oplus T_2'.
	$$
	Since $\CC_1$ and $\CC_2$ have no common
	indecomposable objects up to isomorphism,
	Krull--Schmidt gives $T_i\cong T_i'$.
	By Lemma~\ref{lem:radical-matrices},
	$f_1$ and $f_2$ are isomorphisms.
    The morphism relations give
	$$
	g\circ b=b'\circ K(f_1),\qquad
	C(f_2)\circ\overline a\circ\rho_Z
	=\overline a'\circ\rho_{Z'}\circ g.
	$$
	The first identity shows that $(f_1,g)$ is an isomorphism
	in $\mathcal A$. Since $Q(f_1,g)\circ\rho_Z=\rho_{Z'}\circ g$ and
	$\rho_Z$ is surjective, the second identity gives
	$$
	C(f_2)\circ\overline a
	=\overline a'\circ Q(f_1,g).
	$$
	Therefore $((f_1,g),f_2):U\to U'$ is an isomorphism
	in $\m(Q,C|_{\CC_2})$.
	This proves that $\Theta$ induces a bijection on
	isomorphism classes.
	
	Finally, $\Theta$ is additive and reflects zero objects.
	If $\Theta(U)\cong W_1\oplus W_2$, choose $U_i$ with
	$\Theta(U_i)\cong W_i$. Additivity and injectivity on
	isomorphism classes give $U\cong U_1\oplus U_2$.
	Thus $U$ is indecomposable if and only if $\Theta(U)$ is,
	and the bijection restricts to a bijection on isomorphism
	classes of indecomposable objects.
\end{proof}

We now apply Theorem~\ref{thm:map-finite} to factorization categories.
\begin{theorem}\label{prop:factorization}
	Assume that $\CC$ and $\modu A$ are representation-finite.
	If $\f_A(\mu;G,F)$ satisfies property~$(P)$,
	then $\f_A(\mu;G,F)$ is representation-finite.
\end{theorem}

\begin{proof}
	By Lemma~\ref{lem:normal-form}, it suffices to show that
	$\f=\f_A(0;K,C)$ is representation-finite.
	In fact, the category $\f$ also satisfies property~$(P)$.
	Suppose that two indecomposable objects of $\f$
	have isomorphic components $N\cong N'$ and $Y\cong Y'$.
	Since $I(N)\cong I(N')$, their reconstructions have
	isomorphic middle terms
	$I(N)\oplus Y\cong I(N')\oplus Y'$.
	The reconstructions are isomorphic by property~$(P)$ for
	$\f_A(\mu;G,F)$, so Lemma~\ref{lem:normal-form} shows that
	the two reduced objects are isomorphic.
	We work in $\f$ for the rest of the proof.
	
	We claim that $K(T)=0$ or $C(T)=0$ for every indecomposable
	$T\in\CC$. Otherwise, choose nonzero $D$-linear maps
	$b:K(T)\to D$ and $a:D\to C(T)$. Since $T$ and $D$ are indecomposable in $\CC$ and
	$\modu A$, respectively, and $b,a$ are nonzero, the objects
	$$
	(T,D,b,0),\qquad (T,D,0,a)
	$$
	are indecomposable by Remark~\ref{rem:factorization-basic}.
	Property~$(P)$ gives an isomorphism between them, whose
	middle component $g:D\to D$ is invertible and satisfies
	$g\circ b=0$. This contradicts $b\ne0$.
	
	Let $\CC_1,\CC_2,\mathcal A,Q$ and $\Theta$ be as in
	Lemma~\ref{lem:factorization-comma}.
	Both $\CC_1$ and $\CC_2$ are representation-finite.
	For an indecomposable $Z=(T_1,Y,b)\in\mathcal A$, the object
	$(Z,0,0)$ is indecomposable in $\m(Q,C|_{\CC_2})$.
	By the lemma, $\Theta(Z,0,0)$ is therefore indecomposable
	in $\f$, with components $T_1$ and $Y$.
	If two such objects $Z,Z'$ have isomorphic components,
	property~$(P)$ gives
	$\Theta(Z,0,0)\cong\Theta(Z',0,0)$, so $Z\cong Z'$ by
	Lemma~\ref{lem:factorization-comma}. Since $\CC_1$ and $\modu A$ are
	representation-finite, Theorem~\ref{thm:map-finite} shows
	that $\mathcal A$ is representation-finite.
	
	Now let $U=(Z,T_2,\overline a)$ and
	$U'=(Z',T_2',\overline a')$ be indecomposable objects of
	$\m(Q,C|_{\CC_2})$ with $Z\cong Z'$ and $T_2\cong T_2'$.
	Writing $Z=(T_1,Y,b)$ and $Z'=(T_1',Y',b')$, we have
	$T_1\cong T_1'$ and $Y\cong Y'$.
	Thus $\Theta(U)$ and $\Theta(U')$ are indecomposable objects
	of $\f$ with isomorphic components. By property~$(P)$, $\Theta(U)\cong\Theta(U')$.
	Hence $U\cong U'$ by Lemma~\ref{lem:factorization-comma}.
	Since $\mathcal A$ and $\CC_2$ are representation-finite,
	Theorem~\ref{thm:map-finite} shows that
	$\m(Q,C|_{\CC_2})$ is representation-finite.
	Lemma~\ref{lem:factorization-comma} now implies that $\f$
	is representation-finite.
\end{proof}

\section{Composition factor algebras and finite representation type}\label{sec:structure}

\subsection{Local composition factor algebras and matrix decompositions}
\begin{definition}\label{def:composition-factor-algebra}
An Artin algebra $\Lambda$ is called a \emph{composition factor algebra}
if, for any two indecomposable modules $M,N\in\modu\Lambda$, the equalities
$$
[M:S]=[N:S]\qquad\text{for every simple }\Lambda\text{-module }S
$$
imply $M\cong N$. Here $[M:S]$ denotes the multiplicity of $S$
in a composition series of $M$.
\end{definition}
\begin{remark}\label{rem:invariance}
	The class of composition factor algebras is closed under Morita
	equivalence and passage to opposite algebras. 
	Quotients of composition factor algebras are composition factor algebras
	\cite[Proposition~2.8]{Blasco}. For a two-sided ideal $I$ of $\Lambda$, restriction of scalars along
	$\Lambda\to\Lambda/I$ identifies $\modu(\Lambda/I)$ with the full
	subcategory of $\modu\Lambda$ consisting of modules annihilated by $I$. This identification
	preserves submodules, composition series and direct-sum
	decompositions.
\end{remark}

We begin with an elementary consequence of the definition.

\begin{lemma}\label{lem:same-simple-ends}
Let $\Lambda$ be a composition factor algebra. If
$M\in\modu \Lambda$ has simple top and simple socle, both isomorphic
to $S$, then $M$ is uniserial and every composition factor of
$M$ is isomorphic to $S$.
\end{lemma}

\begin{proof}
We argue by induction on $\len_\Lambda(M)$. The assertion is clear
when $M$ is simple. Otherwise, the simple top makes $\rad M$
the unique maximal submodule, so every proper submodule of
$M$, including $\soc M$, lies in $\rad M$. Consequently,
$$
\soc(\rad M)=\soc M\cong S,
\qquad \topm(M/\soc M)\cong M/\rad M\cong S.
$$
Thus $\rad M$ and $M/\soc M$ are indecomposable by
Fact~\ref{fact:finite-length}.   Both have the composition
factors of $M$ with one copy of $S$ removed. Since $\Lambda$ is
a composition factor algebra, $\rad M\cong M/\soc M$.
In particular, $\rad M$ has simple top and socle, both
isomorphic to $S$. By induction, $\rad M$ is uniserial with
all composition factors isomorphic to $S$. The same is true
of $M$: its top is $S$,   and all its proper submodules lie
in the uniserial module $\rad M$.
\end{proof}

We next consider the local case.

\begin{proposition}\label{prop:local-finite}
Let $\Lambda$ be a local composition factor algebra.
Then $\Lambda$ is Nakayama and hence representation-finite.
\end{proposition}

\begin{proof}
We write $S=\Lambda/J(\Lambda)$. Since $S$ is the unique simple left
$\Lambda$-module up to isomorphism,   additivity of $R$-length gives
$$
\len_R(M)=\len_\Lambda(M)\len_R(S)\qquad(M\in\modu \Lambda).
$$
The regular modules ${}_\Lambda\Lambda$ and $\Lambda_\Lambda$
are indecomposable because $\Lambda$ is local. Matlis duality preserves indecomposability and
$R$-length,   so ${}_\Lambda \Lambda$ and $\mathbb D(\Lambda_\Lambda)$ are
indecomposable left modules of the same composition length.
Since $\Lambda$ is a composition factor algebra, we obtain
$$
{}_\Lambda \Lambda\cong\mathbb D(\Lambda_\Lambda).
$$
Thus $\Lambda$ is self-injective, with
$\topm({}_\Lambda \Lambda)\cong S\cong\soc({}_\Lambda \Lambda)$.
By Lemma~\ref{lem:same-simple-ends}, the module ${}_\Lambda \Lambda$ is uniserial.
Since $\Lambda$ is local and self-injective, it is Nakayama, and hence representation-finite
\cite[VI, Theorem~2.1(b)]{ARS}.
\end{proof}

\begin{notation}\label{not:corner}
For a basic Artin algebra $\Lambda$ and a primitive idempotent $e$, set
$e'=1-e$. The decomposition by these idempotents gives the
natural algebra isomorphism
\begin{equation}\label{eq:localalgebra}
 \Lambda\cong\begin{pmatrix}A&U\\V&B\end{pmatrix},
 \qquad A=e\Lambda e,\quad B=e'\Lambda e',\quad U=e\Lambda e',\quad V=e'\Lambda e.
\end{equation}
The modules $U$ and $V$ are
respectively an $A$--$B$-bimodule and a
$B$--$A$-bimodule.
\end{notation}

Since $\Lambda e$ is indecomposable, its endomorphism ring
$\End_\Lambda(\Lambda e)\cong A^{\op}$ is local \cite[II, \S2]{ARS}.
Thus the corner algebra $A=e\Lambda e$ is local.

Write $S_e=\Lambda e/J(\Lambda)e$ for the simple module corresponding to $e$.
Since $\Lambda$ is basic, $eS_e=S_e$ and $e'S_e=0$.
As a left $A$-module, $S_e$ is naturally isomorphic to $A/J(A)$.
The other simple $\Lambda$-modules are annihilated by $e$ and restrict
to pairwise nonisomorphic simple $B$-modules. Restriction by an
idempotent is exact. Hence, for every $M\in\modu \Lambda$,
\begin{equation}\label{eq:corner-multiplicities}
 [M:S_e]=\len_A(eM),\qquad
 [M:S]=[e'M:e'S]\quad(S\text{ simple},\ S\not\cong S_e).
\end{equation}

The following lemma provides the relations needed for the
factorization description in the next subsection.

\begin{lemma}\label{prop:structure}
Let $\Lambda$ be a basic composition factor algebra, and let $e$
be a primitive idempotent. Use the notation in
\eqref{eq:localalgebra}.   Then we have
\begin{enumerate}
\item[\textup{(1)}] $UV=0$.
\item[\textup{(2)}] $J(A)U=0=VJ(A)$.
\item[\textup{(3)}] $U\otimes_BV=0$.
\end{enumerate}
\end{lemma}

\begin{proof}
(1) Suppose that $0\ne c\in UV\subseteq \Lambda e$. Choose a
left $\Lambda$-submodule $N\subseteq \Lambda e$ maximal among those not
containing $c$, and let $\overline c$ be the image of $c$
in $M=\Lambda e/N$. By maximality, every nonzero submodule of $M$
contains $\overline c$, and hence contains $\Lambda\overline c$.
Choose a simple submodule $T\subseteq M$. Then
$0\ne\Lambda\overline c\subseteq T$, so $T=\Lambda\overline c$.
Thus $\soc M=\Lambda\overline c$ is simple, and
$e\overline c=\overline c\ne0$ gives $\soc M\cong S_e$.
The module $M$ also has top $S_e$, being a nonzero quotient
of the projective module $\Lambda e$ with simple top. By
Lemma~\ref{lem:same-simple-ends}, all its composition factors
are isomorphic to $S_e$.

Exactness of $e'(-)$ now gives $e'M=0$. Thus $V=e'\Lambda e$ is
contained in $N$, and the $\Lambda$-submodule $N$ also contains
$UV$. This contradicts $c\notin N$, proving $UV=0$.

(2) By (1), $UV=0$, so $V$ is a left $\Lambda$-submodule of $\Lambda e$.
The canonical $R$-module isomorphism $\Lambda e/V\cong A$ gives
$A$ a left $\Lambda$-module structure. For
$a=\begin{pmatrix}r&u\\v&b\end{pmatrix}\in \Lambda$, these actions are
$$
a\cdot w=bw\quad(w\in V),\qquad
a\cdot x=rx\quad(x\in A).
$$
In particular, every left $B$-submodule of $V$ inherits a left
$\Lambda$-module structure.

For convenience, we denote $J(A)$ by $J$.
Let $Y$ be an indecomposable direct
summand of the left $B$-module $V$. For
$0\ne q\in\Hom_B(V,Y)$, define a $\Lambda$-action on the
$R$-module $M_q=A\oplus Y$ by
$$
\begin{pmatrix}r&u\\v&b\end{pmatrix}
\begin{pmatrix}x\\y\end{pmatrix}
=\begin{pmatrix}rx\\q(vx)+by\end{pmatrix}.
$$
This defines a $\Lambda$-module structure, since $q$ is a homomorphism
of $B$-modules and, by (1),
$$
UV=0,\qquad (VU)Y\subseteq V(UV)=0.
$$
We claim that $M_q$ is indecomposable. Suppose that
$M_q=M_1\oplus M_2$ with both summands nonzero. Since
$eM_q\cong{}_{A}A$ is indecomposable, we may assume
that $eM_2=0$. Then $e'M_2\ne0$, so the indecomposability of
$e'M_q\cong Y$ gives $e'M_1=0$. It follows that
$$
VM_1\subseteq e'M_1=0,
\qquad VM_2=VeM_2=0.
$$
Since $q\ne0$, choose $v\in V$ with $q(v)\ne0$. Then
$v(e,0)=(0,q(v))\ne0$, contradicting $VM_q=0$.

Moreover, the inclusion of $Y$ and the projection onto $A$
give a short exact sequence of $\Lambda$-modules
$$
0\longrightarrow Y\longrightarrow M_q
\longrightarrow A\longrightarrow0.
$$
Thus the modules $M_q$ have the same composition factors for
fixed $Y$.

Suppose that $VJ\ne0$, and choose $n\ge1$ maximal with
$VJ^n\ne0$.   Such an $n$ exists because $J$ is nilpotent.
Choose an indecomposable direct summand $Y$ of the left
$B$-module $V$ whose projection $p:V\to Y$ satisfies
$p(VJ^n)\ne0$.
There exists $t\in J^n$ such that the homomorphism
$$
h:V\longrightarrow Y,\qquad h(v)=p(vt)
$$
is nonzero. Since $VJ^{n+1}=0$, we have $h(VJ)=0$, whereas
$p(VJ)\ne0$. The action formula therefore gives
$$
(VJ)M_h=0,\qquad (VJ)M_p\ne0.
$$
Consequently, $M_h$ and $M_p$ are nonisomorphic indecomposable
modules with the same composition factors. This contradicts the
assumption that $\Lambda$ is a composition factor algebra.
Hence $VJ(A)=0$. By Remark~\ref{rem:invariance}, $\Lambda^{\op}$
is also a composition factor algebra. Applying the same argument
to $\Lambda^{\op}$ gives
$J(A)U=0$.

(3) Suppose that $U\otimes_BV\ne0$. By additivity of tensor
product,   there is an indecomposable direct summand $Y$ of the
left $B$-module $V$ with $U\otimes_BY\ne0$. Let $p:V\to Y$ be its projection.
Since $A$ is local,   a simple quotient of $U\otimes_BY$
is isomorphic to a simple submodule of ${}_AA$.
Composing the quotient map with such an embedding gives
$$
0\ne\lambda\in\Hom_A(U\otimes_BY,A).
$$
Define a $\Lambda$-module $N_\lambda$ on $A\oplus Y$ by
$$
\begin{pmatrix}r&u\\v&b\end{pmatrix}\cdot
\begin{pmatrix}x\\y\end{pmatrix}
=\begin{pmatrix}rx+\lambda(u\otimes y)\\by\end{pmatrix}.
$$
The relations $UV=0$ and $(VU)Y=0$ ensure that this is a
$\Lambda$-module structure. The same decomposition argument as in
(2) shows that $N_\lambda$ is indecomposable: a nontrivial
decomposition would place its two indecomposable components
$A$ and $Y$ in different summands, forcing the action
of $U$ to be zero, contrary to $\lambda\ne0$.

The module $M_p$ from (2) is also indecomposable. The exact
sequence there and the exact sequence
$$
0\longrightarrow A\longrightarrow N_\lambda
\longrightarrow Y\longrightarrow0
$$
show that $M_p$ and $N_\lambda$ have the same composition factors.
The action of $U$ is zero on $M_p$ and nonzero on
$N_\lambda$, since $\lambda\ne0$. Hence these modules are
not isomorphic. This contradicts the assumption that $\Lambda$
is a composition factor algebra and proves $U\otimes_BV=0$.
\end{proof}

\subsection{Modules as factorizations}
Let $\Lambda$ be a basic composition factor algebra and
$e$ a primitive idempotent. Set $e'=1-e$.
As in \eqref{eq:localalgebra}, write
$$
A=e\Lambda e,\qquad B=e'\Lambda e',\qquad
U=e\Lambda e',\qquad V=e'\Lambda e,
$$
and set $D=A/J(A)$.

For $N\in\modu B$, set
\begin{equation}\label{eq:mudef}
	L(N)=U\otimes_BN,\qquad H(N)=\Hom_B(V,N).
\end{equation}
The bimodule structures give $L(N)$ and $H(N)$ the left
$A$-actions
$$
r(u\otimes n)=(ru)\otimes n,\qquad (rh)(v)=h(vr).
$$
Since $J(A)U=0=VJ(A)$ by Lemma~\ref{prop:structure}(2),
$L(N)$ and $H(N)$ are annihilated by $J(A)$ and hence
are naturally $D$-modules.
For a $B$-module homomorphism $f:N\to N'$, define
$$
L(f)=1_U\otimes f,\qquad H(f)(h)=f\circ h.
$$
These formulas define additive functors
$$
L,H:\modu B\longrightarrow\modu D.
$$

For each $N\in\modu B$, multiplication in $\Lambda$ induces a map
$$
\mu_N:L(N)\longrightarrow H(N),\qquad
\mu_N(u\otimes n)(v)=(vu)n.
$$
Associativity in $\Lambda$ makes this a well-defined
homomorphism of $D$-modules.
For a $B$-module homomorphism $f:N\to N'$,
$B$-linearity gives
$H(f)\circ\mu_N=\mu_{N'}\circ L(f)$.
Thus the maps $\mu_N$ define a natural transformation
$\mu:L\to H$.
We consider the factorization category $\f_A(\mu;L,H)$.

We now define a functor
$\Phi:\modu\Lambda\to\f_A(\mu;L,H)$.
For $M\in\modu\Lambda$, the $\Lambda$-action induces
homomorphisms of $A$-modules
$$
U\otimes_B e'M\xrightarrow{\beta_M}eM
\xrightarrow{\alpha_M}\Hom_B(V,e'M),
$$
where $\beta_M(u\otimes e'm)=um$ and
$\alpha_M(em)(v)=vm$.
These maps satisfy $\alpha_M\circ\beta_M=\mu_{e'M}$.
On objects and $\Lambda$-module homomorphisms
$\varphi:M\to M'$, we set
$$
\Phi(M)=(e'M,eM,\beta_M,\alpha_M),\qquad
\Phi(\varphi)=(\varphi|_{e'M},\varphi|_{eM}).
$$
The $\Lambda$-linearity of $\varphi$ makes the two squares
in Definition~\ref{def:factorization} commute.
Thus these assignments define a functor.

The following proposition is a special case of the construction of
Dlab and Ringel
\cite[\S1, Proposition~1 and the following remark]{DR}.
We include a proof for completeness.

\begin{proposition}\label{prop:module-diagrams}
The functor $\Phi:\modu \Lambda\to\f_A(\mu;L,H)$ is an
equivalence of categories.
\end{proposition}

\begin{proof}
We construct a quasi-inverse $\Psi$. For an object
$T=(N,X,\beta,\alpha)$ of $\f_A(\mu;L,H)$,
define an action on the $R$-module $X\oplus N$ by
$$
\begin{pmatrix}r&u\\v&b\end{pmatrix}\cdot
\begin{pmatrix}x\\n\end{pmatrix}
=\begin{pmatrix}rx+\beta(u\otimes n)\\\alpha(x)(v)+bn\end{pmatrix}.
$$
The element $1_\Lambda=e+e'$ acts as the identity. To verify associativity,
it suffices to check compatibility with the products $VU$ and $UV$:
$$
\alpha(\beta(u\otimes n))(v)=(vu)n,\qquad
\beta(u\otimes\alpha(x)(v))=(uv)x.
$$
The first identity is exactly $\alpha\circ\beta=\mu_N$.
For the second, $(uv)x=0$ by Lemma~\ref{prop:structure}(1),
so it remains to show that $\beta(u\otimes\alpha(x)(v))=0$.
The $A$-linearity of $\alpha$ and $\beta$ and the $B$-linearity
of each $\alpha(x)$ give a homomorphism of left $A$-modules
$$
(U\otimes_BV)\otimes_A X\longrightarrow X,\qquad
(u\otimes v)\otimes x\longmapsto\beta(u\otimes\alpha(x)(v)).
$$
By Lemma~\ref{prop:structure}(3), we have
$(U\otimes_BV)\otimes_A X=0$, and hence
$\beta(u\otimes\alpha(x)(v))=0$.
The action formula therefore defines a $\Lambda$-module
$\Psi(T)$. For a morphism $(f,g):T\to T'$, set
$$
\Psi(f,g)=g\oplus f:X\oplus N\longrightarrow X'\oplus N'.
$$
The relations $g\circ\beta=\beta'\circ L(f)$ and
$H(f)\circ\alpha=\alpha'\circ g$ show that $g\oplus f$ is
$\Lambda$-linear. Hence these assignments define a functor $\Psi$.

For $M\in\modu \Lambda$, the natural map
$$
\Psi(\Phi(M))=eM\oplus e'M\longrightarrow M,\qquad
(x,n)\longmapsto x+n
$$
is a $\Lambda$-module isomorphism. Conversely, under the canonical
identifications $e\Psi(T)\cong X$ and $e'\Psi(T)\cong N$,
the action maps satisfy $\beta_{\Psi(T)}=\beta$ and
$\alpha_{\Psi(T)}=\alpha$.
These identifications are natural in $T$ and give
$\Phi\circ\Psi\cong1_{\f_A(\mu;L,H)}$.
Thus $\Psi$ is a quasi-inverse of $\Phi$.
\end{proof}

\subsection{Proof of Theorem~\ref{thm:main}}
\begin{proof}[Proof of Theorem~\ref{thm:main}]
By Remark~\ref{rem:invariance}, Morita equivalence allows us to
assume that $\Lambda$ is basic. We induct on $\len_R(\Lambda)$ among
basic composition factor $R$-algebras. Every quotient by a nonzero
proper two-sided ideal has smaller $R$-length and is again a basic
composition factor algebra.

If $J(\Lambda)=0$, the algebra is semisimple and representation-finite.
Otherwise $J(\Lambda)$ contains a simple left ideal $S$.   Since $\Lambda$
is basic, $\Lambda/J(\Lambda)$ is a product of division rings, so there
is a primitive idempotent $e$ with $eS=S$.
We use the notation in \eqref{eq:localalgebra} for this
idempotent and set $I=S\Lambda$.
The two-sided ideal $I$ satisfies
$$
0\ne I\subseteq J(\Lambda),\qquad eI=I.
$$
For $e'=1-e$, we have $e'I=0$, so the quotient map induces
an isomorphism
$$
B=e'\Lambda e'\cong e'(\Lambda/I)e'.
$$
We also write $e$ and $e'$ for their images in $\Lambda/I$. By induction, $\Lambda/I$ is representation-finite.

The functor
$$
(\Lambda/I)e'\otimes_B-:\modu B\longrightarrow\modu(\Lambda/I)
$$
is fully faithful: it is left adjoint to
$\Hom_{\Lambda/I}((\Lambda/I)e',-)\cong e'(-)$,
and its unit is the canonical isomorphism
$N\cong e'(\Lambda/I)e'\otimes_BN=B\otimes_BN$.
It therefore preserves indecomposability and reflects
isomorphisms, so $B$ is representation-finite.

By Lemma~\ref{prop:structure}(1), we have $UV=0$, and so the projection
$$
\Lambda\longrightarrow A,\qquad
\begin{pmatrix}a&u\\v&b\end{pmatrix}\longmapsto a
$$
is a surjective algebra homomorphism. Thus $A$ is a composition
factor algebra by Remark~\ref{rem:invariance}.
Since $A$ is local, Proposition~\ref{prop:local-finite}
shows that $\modu A$ is representation-finite.

By Proposition~\ref{prop:module-diagrams}, we have an equivalence
$$
\modu\Lambda\simeq\f_A(\mu;L,H)
$$
for $\CC=\modu B$ and the functors $L,H$ in \eqref{eq:mudef}.
Let $Z=(N,X,\beta,\alpha)$ and $Z'=(N',X',\beta',\alpha')$
be indecomposable objects of $\f_A(\mu;L,H)$ with
$N\cong N'$ and $X\cong X'$.
The corresponding indecomposable $\Lambda$-modules $M,M'$
have $eM\cong X$ and $e'M\cong N$, and similarly for $M'$.
By \eqref{eq:corner-multiplicities}, they have the same composition
factors. Since $\Lambda$ is a composition factor algebra,
$M\cong M'$, and the equivalence gives $Z\cong Z'$.
Thus $\f_A(\mu;L,H)$ satisfies property~$(P)$.
Theorem~\ref{prop:factorization} now shows that $\Lambda$
is representation-finite, completing the induction.
\end{proof}

\subsection{Examples of composition factor algebras}
We collect some classes of composition factor algebras.
By Remark~\ref{rem:invariance}, their quotient algebras are
still composition factor algebras.

\begin{example}\label{ex:composition-factor-algebras}
\begin{enumerate}
\item A local Artin algebra is a composition factor algebra if and
only if it is Nakayama. One direction follows from Proposition~\ref{prop:local-finite}.
For the converse, every
indecomposable module over a Nakayama algebra is uniserial and
is determined by its top and length
\cite[VI, Theorem~2.1(a)]{ARS}. Since the algebra is local, each indecomposable module is
determined by its length. Thus the algebra is a composition factor algebra.

\item Let $Q$ be a quiver of Dynkin type and $k$ a field.
The indecomposable $kQ$-modules are determined by their dimension
vectors. Thus $kQ$ is a composition factor algebra; see \cite[VIII, Corollary~2.4]{ARS}.

\item A \emph{cycle} of length $t\ge1$ in $\modu\Lambda$
is a sequence
$
M_0\longrightarrow M_1\longrightarrow\cdots
\longrightarrow M_t=M_0
$
of nonzero nonisomorphisms between indecomposable modules.
An Artin algebra $\Lambda$ is \emph{representation-directed}
if $\modu\Lambda$ has no cycles. Such algebras are
composition factor algebras \cite[IX, Corollary~2.2]{ARS}. In particular, representation-finite piecewise hereditary algebras
over algebraically closed fields are representation-directed
\cite[Theorem~1.1(iii)]{HappelPH}. For more examples,
see \cite[Chapter~6]{Ringel84}.

\item More generally, if $\modu\Lambda$ has no short cycles, then
$\Lambda$ is a composition factor algebra
\cite[IX, Corollary~2.3]{ARS}. A short cycle is a cycle of length at most two.

\item Let $k$ be a field and $Q_n$ the oriented cycle with $n\ge3$
vertices. We write $J$ for the arrow ideal of $kQ_n$. For $2\le m<n$, the
Nakayama algebra $kQ_n/J^m$ is a composition factor algebra.
Indeed, every indecomposable module has length at most $m<n$.
Its composition factors correspond to a proper segment of the
cycle, so they determine its top and length and hence its
isomorphism class \cite[VI, Theorem~2.1(a)]{ARS}.
\end{enumerate}
\end{example}

\medskip
\noindent\textbf{Acknowledgements.}
{\itshape\emergencystretch=2em
	The work is supported by National Natural Science Foundation
	of China Grant No. 12371034.
	\par}

\medskip
\noindent\textbf{Declaration of AI use.}
The main ideas and overall approach of this paper are due to the authors.
ChatGPT (OpenAI) was used to assist with the proof of
Proposition~\ref{prop:map-positivity}\textup{(2)}.
The tool was also used to search for relevant literature and
check arguments and references.
The authors take full responsibility for the content of this paper.


\textbf{Bo Le}\\
Department of Mathematical Sciences, Tsinghua University, Beijing, 100084, People's Republic of
China.\\
E-mail: \textsf{bole\_math@163.com}\\[0.3cm]
\textbf{Bin Zhu}\\
Department of Mathematical Sciences, Tsinghua University, Beijing, 100084, People's Republic of
China.\\
E-mail: \textsf{zhu-b@mail.tsinghua.edu.cn}

\begin{thebibliography}{99}
\bibitem{AM}
M.~F.~Atiyah and I.~G.~Macdonald,
\emph{Introduction to Commutative Algebra},
Addison-Wesley, Reading, MA, 1969.

\bibitem{AR}
M.~Auslander and I.~Reiten,
\emph{\href{https://doi.org/10.1215/ijm/1256045730}{Modules determined by their composition factors}},
Illinois J. Math. \textbf{29} (1985), no.~2, 280--301.

\bibitem{ARS}
M.~Auslander, I.~Reiten and S.~O.~Smal\o,
\emph{Representation Theory of Artin Algebras},
Cambridge Studies in Advanced Mathematics, vol.~36,
Cambridge University Press, Cambridge, 1995.

\bibitem{Blasco}
V.~Blasco,
\emph{\href{https://doi.org/10.1080/00927872.2025.2585403}{On Artin algebras whose indecomposable modules are determined
by composition factors}},
Comm. Algebra \textbf{54} (2026), no.~7, 2796--2810.

\bibitem{CCWX}
L.~Chen, Z.~Chen, Y.~Wen and W.~Xia,
\emph{Composition factors and finite representation type},
preprint (2026), arXiv:2610.38813.


\bibitem{DR}
V.~Dlab and C.~M.~Ringel,
\emph{\href{https://www.numdam.org/item/CM_1989__70_2_155_0/}{A construction for quasi-hereditary algebras}},
Compositio Math. \textbf{70} (1989), no.~2, 155--175.

\bibitem{First}
U.~A.~First,
\emph{\href{https://doi.org/10.1080/00927872.2015.1065836}{Categorical realizations of quivers}},
Comm. Algebra \textbf{44} (2016), no.~6, 2567--2582.

\bibitem{Happel}
D.~Happel,
\emph{\href{https://doi.org/10.1090/S0002-9939-1982-0663860-4}{Composition factors for indecomposable modules}},
Proc. Amer. Math. Soc. \textbf{86} (1982), no.~1, 29--31.

\bibitem{HappelPH}
D.~Happel,
\emph{\href{https://doi.org/10.4064/cm-78-2-261-266}{Hochschild cohomology of piecewise hereditary algebras}},
Colloq. Math. \textbf{78} (1998), no.~2, 261--266.

\bibitem{Krause}
H.~Krause,
\emph{\href{https://doi.org/10.1016/j.exmath.2015.10.001}{Krull--Schmidt categories and projective covers}},
Exposition. Math. \textbf{33} (2015), no.~4, 535--549.

\bibitem{Lang}
S.~Lang,
\emph{\href{https://doi.org/10.1007/978-1-4613-0041-0}{Algebra}},
revised third ed., Graduate Texts in Mathematics, vol.~211,
Springer-Verlag, New York, 2002.

\bibitem{MacLane}
S.~Mac Lane,
\emph{Categories for the Working Mathematician}, second ed.,
Graduate Texts in Mathematics, vol.~5,
Springer-Verlag, New York, 1998.

\bibitem{MV}
R.~MacPherson and K.~Vilonen,
\emph{\href{https://doi.org/10.1007/BF01388812}{Elementary construction of perverse sheaves}},
Invent. Math. \textbf{84} (1986), no.~2, 403--435.

\bibitem{Ringel84}
C.~M.~Ringel,
\emph{\href{https://doi.org/10.1007/BFb0072870}{Tame Algebras and Integral Quadratic Forms}},
Lecture Notes in Mathematics, vol.~1099,
Springer-Verlag, Berlin, 1984.

\bibitem{Rotman}
J.~J.~Rotman,
\emph{\href{https://doi.org/10.1007/978-1-4612-4176-8}{An Introduction to the Theory of Groups}}, fourth ed.,
Graduate Texts in Mathematics, vol.~148,
Springer-Verlag, New York, 1995.

\bibitem{Simson}
D.~Simson,
\emph{\href{https://www.math.okayama-u.ac.jp/mjou/mjou1-46/mjou_pdf/mjou_35/mjou_35_089.pdf}{On bimodule matrix problems and artinian bipartite piecewise
peak PI-rings of finite prinjective module type}},
Math. J. Okayama Univ. \textbf{35} (1993), 89--138.


\end{thebibliography}
\end{document}